\documentclass[a4paper,english]{article}
\usepackage[T1]{fontenc}
\usepackage[latin9]{inputenc}
\usepackage{babel}
\usepackage{array}
\usepackage{verbatim}
\usepackage{amsmath}
\usepackage{amsthm}
\usepackage{stmaryrd}
\usepackage{setspace}
\usepackage[unicode=true,pdfusetitle,
 bookmarks=true,bookmarksnumbered=false,bookmarksopen=false,
 breaklinks=false,pdfborder={0 0 1},backref=false,colorlinks=false]
 {hyperref}
\hypersetup{
 final,linkcolor=darkred,urlcolor=darkred,citecolor=darkred}

\makeatletter

\pdfpageheight\paperheight
\pdfpagewidth\paperwidth

\providecommand{\tabularnewline}{\\}

\newcommand{\lyxdeleted}[3]{}

\theoremstyle{plain}
\newtheorem{thm}{\protect\theoremname}
  \theoremstyle{remark}
  \newtheorem{rem}[thm]{\protect\remarkname}
 \theoremstyle{definition}
 \newtheorem*{defn*}{\protect\definitionname}
  \theoremstyle{plain}
  \newtheorem{prop}[thm]{\protect\propositionname}
  \theoremstyle{definition}
  \newtheorem{defn}[thm]{\protect\definitionname}
  \theoremstyle{plain}
  \newtheorem{lem}[thm]{\protect\lemmaname}
  \theoremstyle{plain}
  \newtheorem{cor}[thm]{\protect\corollaryname}
  \theoremstyle{definition}
  \newtheorem{example}[thm]{\protect\examplename}

\usepackage{xcolor}
\usepackage[hang,flushmargin]{footmisc} 
\usepackage{hyperref}
\hypersetup{colorlinks=true, linkcolor=red, urlcolor=red, citecolor=red}
\usepackage{cite} 
\usepackage{charter}
\usepackage[bitstream-charter,cal=cmcal]{mathdesign}
\usepackage{setspace}
\usepackage{bm} 
\usepackage[absolute,overlay]{textpos} 
\usepackage{listings}
\usepackage{cite}

\usepackage{tikz}
\usetikzlibrary{arrows,positioning,fit}
\tikzset{    
>=stealth',
punkt/.style={rectangle, rounded corners, draw=black, thick, text width=3.5em, minimum height=2em, text centered},
empty/.style={rectangle, rounded corners, text width=3.5em, minimum height=2em, text centered},
bigpunkt/.style={rectangle, rounded corners, draw=black, thick, text width=6em, minimum height=2em, text centered},
hpunkt/.style={rectangle, rounded corners=10pt, draw=black, thick, text width=6em, minimum height=2em, text centered},
pil/.style={->, thick, shorten <=2pt, shorten >=2pt,} 
every loop/.style={max distance=10mm,looseness=10}
}

\AtBeginDocument{
  
}

\makeatother

  \providecommand{\corollaryname}{Corollary}
  \providecommand{\definitionname}{Definition}
  \providecommand{\examplename}{Example}
  \providecommand{\lemmaname}{Lemma}
  \providecommand{\propositionname}{Proposition}
  \providecommand{\remarkname}{Remark}
\providecommand{\theoremname}{Theorem}

\begin{document}
\global\long\def\ms{\boldsymbol{X}_{\boldsymbol{\mathcal{L}}_{\Lambda}}}
\global\long\def\ll{\boldsymbol{\mathcal{L}}_{\Lambda}}
\global\long\def\td{\mathcal{T}_{D}}
\global\long\def\md{\boldsymbol{X}_{D}}

\title{Metrics for Formal Structures, with an Application to Kripke Models
and their Dynamics}

\author{Dominik Klein\thanks{Department of Philosophy, Bayreuth University, and Department of Political Science,
University of Bamberg} ~and Rasmus K. Rendsvig\thanks{LUIQ, Theoretical Philosophy, Lund University, and Center for Information and Bubble Studies, University of Copenhagen}\date{}}

\maketitle
\noindent \setstretch{1}\begin{abstract}

\noindent This paper introduces and investigates a family of metrics
on sets of structures for formal languages, with a special focus on
their application to sets of pointed Kripke models and modal logic,
and, in extension, to dynamic epistemic logic. The metrics are generalizations
of the Hamming distance applicable to countably infinite binary strings
and, by extension, logical theories or semantic structures. We first
study the topological properties of the resulting metric spaces. A
key result provides sufficient conditions for spaces having the Stone
property, i.e., being compact, totally disconnected and Hausdorff.
Second, we turn to mappings, where it is shown that a widely used
type of model transformations, product updates, give rise to continuous
maps in the induced topology.\vspace{14pt}

\noindent \textbf{Keywords:} metric space, general topology, modal
logic, Kripke model, model transformation, dynamic epistemic logic.

\noindent 
\end{abstract}\setstretch{1.25}

\section{Introduction}

This paper introduces and investigates a family of metrics on spaces
of a graph type, namely pointed Kripke models. Intuitively, a metric
is a distance measuring function: a map that assigns a positive, real
value to pairs of elements of some set, specifying how far these elements
are from one another. We present a general way of assigning such numbers
to pointed Kripke models, the most widely used semantic structures
for modal logic.\footnote{The metrics introduced are equally applicable to other semantic structures,
e.g., neighborhood models, as is shown below. We focus on Kripke models
due to their widespread use and tight connection with dynamic epistemic
logic.}

Apart from mathematical interest, there are several motivations for
having a metric between pointed Kripke models, including applications
in iterated multi-agent belief revision in the style of \cite{aucher2010generalizing,Caridroit2016}
and the application of dynamical systems theory to information dynamics
modeled using dynamic epistemic logic \cite{Benthem_OneLonely,Sadzik2006,Rendsvig-DS-DEL-2015,Benthem-DS-2016}.
We will expand on these applications, together with the connections
to this literature, in a later version of this paper. \medskip{}

Metrics on sets of pointed Kripke models exist have previously been
introduced. To the best of our knowledge, the first such was introduced
by G. Aucher in his \cite{aucher2010generalizing} for the purpose
of generalizing AGM to a multi-agent setting. For a similar purpose,
the authors of \cite{Caridroit2016} introduce 6 different metrics.
Neither investigate the topological properties of their metrics, but
we look forward to, in latter work, performing an in-depth comparison.

\medskip{}

This paper progresses as follows. In Section \ref{sec:Hamming}, we
introduce a family of metrics on infinite strings and present a general
case for applying the metrics to arbitrary sets of structures, given
that the structures are abstractly described by a countable set and
a possibly multi-valued semantics. We show how the metrics may be
applied to sets of pointed Kripke models and gives examples of metrics
natural from a modal logical point of view. Section \ref{sec:Topo}
is on topological properties of the resulting spaces. We show that
the introduced metrics all induce the Stone topology, which is shown
totally disconnected and, under restrictions, compact. In Section
\ref{sec:Maps}, we turn to mappings. In particular we investigate
a widely used family of mappings defined using a particular graph
product (product update with action models). We show the family continuous
with respect to the Stone topology.
\begin{rem}
This paper is not self-contained. Only definitions for a selection
of standard terms are included, and are so to fix notation. For here
undefined notions from modal logic, refer to e.g. \cite{BlueModalLogic2001,ModelTheoryModalLogic}.
For topological notions, refer to e.g. \cite{Munkres}.
\end{rem}

\section{Generalizing the Hamming Distance\label{sec:Hamming}}

The method we propose for measuring distance between pointed Kripke
models is a particular instantiation of a more general approach. The
more general approach concerns measuring the distance between finite
or infinite strings taking values from some set, $V$. The set $V$
may be thought of as containing the possible truth values for some
logic. For normal modal logic, $V$ would be binary, and the resulting
strings be made, e.g., of 1s and 0s. We think of pointed Kripke models
as being represented by such countably infinite strings: A model's
string will have a 1 on place $k$ just in case the model satisfies
the $k$th formula in some enumeration of the modal language, $0$
else.\footnote{This is the intuition. Details are in Section \ref{sec:DistancesKripke}:
To avoid double-counting, the propositions of the language \emph{modulo}
logical equivalence for a suited logic is used. \label{fn:not-formulas}}

A distance on sets of finite strings of a fixed length has been known
since 1950, when it was introduce by R.W. Hamming \cite{Hamming1950}.
Informally, the \textbf{Hamming distance} between two such strings
is the number of places on which the two strings differ. If the strings
are infinite, the Hamming distance between them clearly is sometimes
undefined.

For faithfully representing pointed Kripke models as strings of formulas,
the strings in general needs to be infinite. This is the case as there
are infinitely many modally expressible mutually non-equivalent properties
of pointed Kripke models. We return to this below. To accommodate
infinite strings, we generalize the Hamming distance:\footnote{ To the best of our knowledge, the generalization is new\textemdash at
least we have failed to find it in the comprehensive \emph{Encyclopedia
of Distances} \cite{Deza2016}.}
\begin{defn*}
\label{def:string-metric}Let $S$ be a set of strings over a set
$V$ such that either $S\subseteq V^{n}$ for some $n\in\mathbb{N}$,
or for all $s\in S$, for all $i\in\mathbb{N}$, $s_{i}\in V$. For
all $k\in\mathbb{N}$, let

\setstretch{.8}
\[
d_{k}(s,s')={\textstyle \begin{cases}
0 & \text{ if }s_{k}=s'_{k}\\
1 & \text{ else}
\end{cases}}
\]
\setstretch{1.25}Let $w:\mathbb{N}\rightarrow\mathbb{R}_{>0}$ assign
a strictly positive \textbf{weight} to each natural number such that
$(w(k))_{k\in\mathbb{N}}$ form a convergent series, i.e., $\sum_{k=1}^{\infty}w(k)<\infty$. 

The function $d_{w}:S\times S\longrightarrow\mathbb{R}$ is then defined
by, for each $s,s'\in S$
\[
d_{w}(s,s')=\sum_{k=1}^{\infty}w(k)d_{k}(s,s').
\]
\end{defn*}
\begin{prop}
\label{prop:metric}Let $S$ and $d_{w}$ be as above. Then $d_{w}$
is a metric on $S$. 

\begin{proof}
Each $d_{w}$ is a metric on $S$ as it for all $s,s',s''\in X$ satisfies

\noindent \textbf{Positivity}, $d_{w}(s,s')\geq0$: The sum defining
$d_{w}$ contains only non-negative terms. 

\noindent \textbf{Identity of indiscernibles}, $d_{w}(s,s')=0$ iff
$s=s'$: $d_{w}(s,s')=0$ iff $d_{k}(s,s')=0$ for all $k$ iff $s_{k}=s'_{k}$
for all $k$ iff $s=s'$.

\noindent \textbf{Symmetry}, $d_{w}(s,s')=d_{w}(s',s)$: As $d_{k}(s,s')=d_{k}(s',s)$
for all $k$. 

\noindent \textbf{Triangular inequality}, $d_{w}(s,s'')\leq d_{w}(s,s')+d_{w}(s',s'')$:
If $s$ and $s''$ differ on any position $k$, then either $s$ and
$s'$ or $s'$ and $s''$ have to differ on the same position. Hence
$w(k)d_{k}(s,s'')\leq w(k)d_{k}(s,s')+w(k)d_{k}(s',s'')$, for each
$k$, which establishes the triangular equality: $\sum_{k=1}^{\infty}w(k)d_{k}(s,s'')\leq\sum_{k=1}^{\infty}w(k)d_{k}(s,s')+\sum_{k=1}^{\infty}w(k)d_{k}(s',s'')$.
\end{proof}

\end{prop}

\begin{rem}
The Hamming distance is a special case of the defined family. For
$S\subseteq\mathbb{R}^{n}$, the Hamming distance $d_{H}$ is defined,
cf. \cite{Deza2016}, by $d_{H}(s,s')=|\{i:1\leq i\leq n,s_{i}\not=s'_{i}\}|$.
This function is a member of the above family given by the weight
function $h(k)=1$ for $1\leq k\leq n$, $h(k')=0$ for $k'>n$.
\end{rem}

\section{Metrics for Formal Structures\label{sec:MetricsFormalStr}}

The metrics defined above may be indirectly applied to any set of
structures that serves as a valuating semantics for a countable language.
In essence, what is required is simply an assignment of suitable weights
to formulas of the language and an addition of the weights of formulas
on which structures differ in valuation.

To illustrate the generality of the approach, we initially take the
following inclusive view on semantic valuation:
\begin{defn}
\label{def:valuation}Let a \textbf{valuation} be any map $\nu:X\times D\longrightarrow V$
where $X,D$ and $V$ are arbitrary sets, but $D$ required countable.
Refer to elements of $X$ as \textbf{structures}, to $D$ as the \textbf{descriptor},
and to elements of $V$ as \textbf{values}.
\end{defn}

A valuation $\nu$ assigns a value from $V$ to every pair $(x,\varphi),x\in X,\varphi\in D$.
The valuation Jointly, $\nu$ and $X$ thus constitute a $V$-valued
semantics for the descriptor $D$. The term \emph{descriptor} is used
here and below to emphasize the potential \emph{lack of grammar} in
the set $D$. The descriptor \emph{may} be a formal language, but
it is not required. In particular, the descriptor may be a \emph{strict
subset} of a formal language, containing only formulas of special
interest. This is exemplified in Section \ref{subsec:Examples}.

Two structures in $X$ may be considered equivalent by $\nu$, i.e.,
be assigned identical values for all $\varphi\in D$. To avoid that
two non-identical, but semantically equivalent, structures receive
a distance of zero (and thus violate the requirements of a metric),
metrics are defined over suitable quotients:
\begin{defn*}
Given a valuation $\nu:X\times D\longrightarrow V$ and a subset $D'$
of $D$, denote by $\boldsymbol{X}_{D'}$ the \textbf{quotient of
$X$ under $D'$ equivalence}, i.e., $\boldsymbol{X}_{D'}=\{\boldsymbol{x}{}_{D'}\colon x\in X\}$
with $\boldsymbol{x}{}_{D'}=\{y\in X\colon\nu(y,\varphi)=\nu(x,\varphi)\text{ for all }\varphi\in D'\}$.
\end{defn*}
Quotients are defined for subsets $D'$ of $D$ in accordance with
the comment concerning the term \emph{descriptor} above: For some
structures, it may be natural to define a semantics for a complete
formal language, $\mathcal{L}$. However, if only a subset $D'\subseteq\mathcal{L}$
is deemed relevant in determining distance, it is natural to focus
on structures under $D'$ equivalence. The terminological usage is
consistent as the subset $D'$ is itself a descriptor for the restricted
map $\nu_{|X\times D'}$.

Finally, we obtain a family of metrics on a quotient $\boldsymbol{X}_{D}$
in the following manner:
\begin{defn*}
\label{def:structure-metric} Let $\nu:X\times D\longrightarrow V$
be a valuation and $\varphi_{1},\varphi_{2},...$ an enumeration of
$D$. For all $x,y\in X$, all $\boldsymbol{x},\boldsymbol{y}\in\boldsymbol{X}_{D}$
and all $k\in\mathbb{N}$, let

\setstretch{.8}
\[
d_{k}(\boldsymbol{x},\boldsymbol{y})={\textstyle \begin{cases}
0 & \text{ if }\nu(x,\varphi_{k})=\nu(y,\varphi_{k})\\
1 & \text{ else}
\end{cases}}
\]
\setstretch{1.25}Call $w:D\rightarrow\mathbb{R}_{>0}$ a \textbf{weight
function} if it assigns a strictly positive \textbf{weight} to each
$\varphi\in D$ such that $(w(\varphi_{k}))_{k\in\mathbb{N}}$ produce
a convergent series. 

The function $d_{w}:\boldsymbol{X}_{D}\times\boldsymbol{X}_{D}\rightarrow\mathbb{R}$
is then defined by, for each $\boldsymbol{x},\boldsymbol{y}\in\boldsymbol{X}_{D}$
\[
d_{w}(\boldsymbol{x},\boldsymbol{y})=\sum_{k=1}^{\infty}w(\varphi_{k})d_{k}(\boldsymbol{x},\boldsymbol{y}).
\]
The set of such maps $d_{w}$ is denoted $\mathcal{D}_{(X,\nu,D)}$.
\end{defn*}
\begin{prop}
\label{prop:structure-metric} Every $d_{w}\in\mathcal{D}_{(X,\nu,D)}$
is a metric on $\boldsymbol{X}_{D}$.
\end{prop}

\begin{proof}
That $d_{w}$ is a metric on $\boldsymbol{X}_{D}$ is argued using
\ref{prop:metric}: Define $S$ as the set of length $|D|$ strings
over $V$ given by $S=\{s_{\boldsymbol{x}}:\boldsymbol{x}\in\boldsymbol{X}_{D}\}$
such that for each $\boldsymbol{x}\in\boldsymbol{X}_{D}$, for each
$\varphi_{i}\in D$, $s_{\boldsymbol{x},i}=\nu(x,\varphi_{i})$. Then
the map $f:\boldsymbol{X}_{D}\rightarrow S$ given by $f(\boldsymbol{x})=s_{\boldsymbol{x}}$
is a bijection. Let $w':\mathbb{N}\rightarrow\mathbb{R}_{>0}$ be
given by $w'(k)=w(\varphi_{k})$ for all $k\in\mathbb{N}$, and let
$d_{w'}$ be the metric on $S$ given by $w'$ cf. Prop. \ref{prop:metric}.
Then $d_{w}(\boldsymbol{x},\boldsymbol{y})=d_{w'}(s_{\boldsymbol{x}},s_{\boldsymbol{y}})$
for all $x,y\in X$. Hence $d_{w}$ is a metric on $\boldsymbol{X}_{D}$.
\end{proof}
\begin{rem}
The choice of descriptor affect both the coarseness of the space $\boldsymbol{X}_{D}$
as well as the metrics definable. We return to this point several
times below.
\end{rem}

\begin{rem}
\label{rem:To-fix-intuitions}To fix intuitions, descriptors have
hitherto been hinted at as being sets of formulas from some \emph{language}.
When interested in metrics that reflect the properties of some \emph{logic},
i.e., not the syntactically discernible formulas, but the logically
discernible propositions, it is natural to partition the language
according to logical equivalence and use the resulting quotient \textendash{}
or a subset thereof \textendash{} as descriptor. This is the approach
pursued here (cf. fn. \ref{fn:not-formulas}).
\end{rem}

\section{The Application to Pointed Kripke Models\label{sec:DistancesKripke}}

To apply the metrics to pointed Kripke models, we follow the above
approach. The set $X$ will be a set of pointed Kripke models and
$D$ a set of modal logical formulas. Interpreting the latter over
the former using standard modal logical semantics gives rise to a
binary set of values, $V$, and a valuation function $\nu:X\times D\longrightarrow V$
that is classic interpretation of modal formulas on Kripke models.
In the following, we will omit all references to $\nu$ , writing$\mathcal{D}_{(X,D)}$
for $d_{w}\in\mathcal{D}_{(X,\nu,D)}$. 

\subsection{Pointed Kripke Models, their Language and Logics}

Let be given a \textbf{signature} consisting of a countable, non-empty
set of propositional \textbf{atoms} $\Phi$ and a countable, non-empty
set of \textbf{operator indices}, $\mathcal{I}$. Call the signature
\textbf{finite} when both $\Phi$ and $\mathcal{I}$ are finite. The
\textbf{modal language} \emph{$\mathcal{L}$ }for $\Phi$ and $\mathcal{I}$
is given by 
\[
\varphi:=\top\;|\;p\;|\;\neg\varphi\;|\;\varphi\wedge\varphi\;|\;\square_{i}\varphi
\]
The language $\mathcal{L}$ is countable. 

A \textbf{Kripke model} for $\Phi$ and $\mathcal{I}$ is a tuple
$M=(\left\llbracket M\right\rrbracket ,R,\left\llbracket \cdot\right\rrbracket )$
where

\noindent %
\begin{tabular}{>{\centering}p{0.03\textwidth}>{\raggedright}p{0.9\textwidth}}
\noalign{\vskip5pt}
 & $\left\llbracket M\right\rrbracket $ is a countable, non-empty set
of \textbf{states};\tabularnewline
\noalign{\vskip3pt}
 &  $R:\mathcal{I}\rightarrow\mathcal{P}(\left\llbracket M\right\rrbracket \times\left\llbracket M\right\rrbracket )$
assigns to each $i\in\mathcal{I}$ an \textbf{accessibility relation}
$R(i)$;\tabularnewline
\noalign{\vskip3pt}
 & $\left\llbracket \cdot\right\rrbracket :\Phi\rightarrow\mathcal{P}(\left\llbracket M\right\rrbracket )$
is an \textbf{atom valuation}, assigning to each atom a set of states.\tabularnewline
\noalign{\vskip3pt}
\end{tabular}

\noindent A pair $(M,s)$ with $s\in\left\llbracket M\right\rrbracket $
is a \textbf{pointed Kripke model}. For the pointed Kripke model $(M,s)$,
the shorter notation $Ms$ is used. For $R(i)$, we write $R_{i}$.

The modal language is evaluated over pointed Kripke models with standard
semantics:
\begin{center}
\begin{tabular}{lll}
$Ms\models p$ & iff & $s\in\left\llbracket p\right\rrbracket $, for all $p\in\Phi$\tabularnewline
\noalign{\vskip3pt}
$Ms\models\neg\varphi$ & iff &  $Ms\not\models\varphi$\tabularnewline
\noalign{\vskip3pt}
$Ms\models\varphi\wedge\psi$ & iff & $Ms\models\varphi$ and $Ms\models\psi$\tabularnewline
\noalign{\vskip3pt}
$Ms\models\square_{i}\varphi$ & iff & for all $t$, $sR_{i}t$ implies $Mt\models\varphi$\tabularnewline
\noalign{\vskip3pt}
\end{tabular}
\par\end{center}

Modal logics may be formulated in $\mathcal{L}$. In this article,
we only use a \textbf{logic }$\Lambda$ we refer only to extensions
of the \textbf{normal modal logics} over the language $\mathcal{L}$.
With $\Lambda$ given by context, let $\boldsymbol{\varphi}$ be the
set of formulas $\Lambda$-provably equivalent to $\varphi$. Denote
the resulting partition $\{\boldsymbol{\varphi}:\varphi\in\mathcal{L}\}$
of $\mathcal{L}$ by $\boldsymbol{\mathcal{L}}_{\Lambda}$.\footnote{$\boldsymbol{\mathcal{L}_{\Lambda}}$ is isomorphic to the domain
of the \textbf{Lindenbaum algebra} of $\Lambda$. For more on the
Lindenbaum algebra and relations to modal logic, see e.g. \cite[pp. 271]{BlueModalLogic2001}} Call $\boldsymbol{\mathcal{L}}_{\Lambda}$'s elements $\Lambda$\textbf{-propositions}.

\subsection{Descriptors for Pointed Kripke Models}

As descriptors for pointed Kripke models, we use sets of $\Lambda$-propositions.
In doing so, the contribution to the distance between two models given
by disagreeing on the truth value of some formula $\varphi\in\mathcal{L}$
will simply be $w(\boldsymbol{\varphi})$ for $\boldsymbol{\varphi}\in\boldsymbol{\mathcal{L}}_{\Lambda}$.
The alternative would be to use sets of $\mathcal{L}$-formulas directly.
This however requires either picking descriptors containing no two
equivalent formulas, or suffering double-counting. We find the suggested
most appealing.
\begin{defn*}
Let $X$ be a set of pointed Kripke models and let $\Lambda$ be a
logic sound with respect to $X$. Then a \textbf{descriptor for }$X$
is any set $D\subseteq\boldsymbol{\mathcal{L}}_{\Lambda}$.
\end{defn*}
\begin{rem}
\label{rem:Soundness}The requirement that $\Lambda$ be sound with
respect to $X$ is needed to ensure the metrics well-defined: It ensures
that for all $x\in X$, if $x\models\varphi$, then for all $\varphi'\in\boldsymbol{\varphi}$,
$x\models\varphi'$. I.e., $x$ cannot be in disagreement with itself
about the valuation of $\boldsymbol{\varphi}$.
\end{rem}

The choice of descriptor has implications on which $\Lambda$-propositions
are taken into account for the metric. Chosing e.g. the set of atomic
propositions as restrictor, will result in a rather coarse perspective.
We will be particularly interested in descriptors that have the same
expressive power as $\mathcal{L}$ (or $\boldsymbol{\mathcal{L}}_{\Lambda}$)
itself:
\begin{defn*}
\label{def:Repr.Descriptor} Say that $D\subseteq\boldsymbol{\mathcal{L}}_{\Lambda}$
is $\Lambda$\textbf{-representative} if, for every $\varphi\in\boldsymbol{\mathcal{L}}$,
there is a set $\{\boldsymbol{\psi}_{i}\}_{i\in I}\subseteq D$ such
that for all sets $S=\{\psi_{i}\}_{i\in J}\cup\{\neg\psi_{i}\}_{i\in I\setminus J}$
with $J\subseteq I$ either $\varphi$ or $\neg\varphi$ is $\Lambda$-entailed
by $S$.
\end{defn*}
The main implication of a descriptor being representative is given
in Lemma \ref{lem:Repr.Descriptor} below. A strict subset of $\boldsymbol{\mathcal{L}}_{\Lambda}$
which is $\Lambda$-representative is presented in Example \ref{ex:n-bisim.metric}. 

\subsection{Modal Spaces}

As stated in Section \ref{sec:MetricsFormalStr}, we construct metrics
on sets of structures \emph{modulo} logical equivalence. The choice
to use a proof-theoretic over a semantic quotient is motivated by
general applicability: The notion of a sound logic in a language evaluated
over a set of structures is conceptually uniform, while the semantic
concept characterizing structural identity suited to the language
in question may be highly variable.\footnote{Compare e.g. isomorphism as an identity concept for first-order languages
with bisimulation suited for standard modal languages and again with
the many specialized versions of bisimulation suited to non-standard
modal languages. See also Example \ref{ex:n-bisim.metric}.} 

In so doing, we follow \cite{KleinRendsvig-LORI2017} in referring
to \emph{modal spaces}:
\begin{defn*}
With $X$ a set of pointed Kripke models and $D$ a descriptor for
$X$,\linebreak{}
 the $D$\textbf{-modal space of $X$ }is denoted $\boldsymbol{X}_{D}$
and is the set $\{\boldsymbol{x}_{D}\colon x\in X\}$ with $\boldsymbol{x}_{D}=\{y\in X:\forall\boldsymbol{\varphi}\in D,y\models\varphi\text{ iff }x\models\varphi\}$.
\end{defn*}
\noindent The subscript of $\boldsymbol{x}_{D}$ is omitted when the
descriptor is clear from context.\medskip{}

The choice of descriptor influence the resulting modal space: $\boldsymbol{X}_{D}$
may be a more or less coarse partition of $X$, with two extremes:
If the descriptor is $\mathcal{L}_{\Lambda}$, the finest partition
is achieved: $\boldsymbol{X}_{\mathcal{L}_{\Lambda}}$, the quotient
of $X$ under $\Lambda$-equivalence. For the coarsest partition,
choose $\{\top\}$ as descriptor: $X_{\{\top\}}$ is simply $\{X\}$.

We are mainly interested in modal spaces that retain the structure
of $X$ as seen by a logic $\Lambda$, i.e., $\boldsymbol{X}_{\mathcal{L}_{\Lambda}}$.
This does not entail that $\mathcal{L}_{\Lambda}$ is the only descriptor
of interest. Others are sufficient:
\begin{lem}
\label{lem:Repr.Descriptor} If $D\subseteq\mathcal{L}_{\Lambda}$
is a $\Lambda$-representative descriptor for $X$, then $\boldsymbol{X}_{D}$
is identical to $\boldsymbol{X}_{\mathcal{L}_{\Lambda}}$, i.e., for
all $x,y\in X$, $y\in\boldsymbol{x}_{D}$ iff $y\in\boldsymbol{x}_{\mathcal{L}_{\Lambda}}$.
\end{lem}

\begin{proof}
 We first show that $y\in\boldsymbol{x}_{D}$ entails $y\in\boldsymbol{x}_{\mathcal{L}_{\Lambda}}$.
Assume $y\in\mathbf{x}_{D}.$ To show that $y\in\boldsymbol{x}{}_{\mathcal{L}_{\Lambda}}$.
we need to prove that for all $\varphi\in\Lambda$ holds $x\vDash\varphi\Leftrightarrow y\vDash\varphi$.
We only show the left-to-right implication, the other direction being
similar. Assume $x\vDash\varphi$. Since D is representative, there
is a set $\{\boldsymbol{\psi}_{i}\}_{i\in I}\subseteq D$ such that
$\{\chi_{i}\}_{i\in I}\vDash\varphi$ where$\chi_{i}=\psi_{i}$ iff
$x\vDash\psi_{i}$and $\chi_{i}=\neg\psi_{i}$ else. Since $y\in\mathbf{x}_{D}$
and $\boldsymbol{\psi}_{i}\in D$, we have$y\vDash\chi_{i}$ for all
$i\in I$. Hence also $y\vDash\varphi$. 

Next we show that that $y\in\boldsymbol{x}_{\mathcal{L}_{\Lambda}}$
entails $y\in\boldsymbol{x}_{D}$ . Assume $y\in\boldsymbol{x}{}_{\mathcal{L}_{\Lambda}}\in X\boldsymbol{}{}_{\mathcal{L}_{\Lambda}}$
In hence holds that $x\vDash\varphi\Leftrightarrow y\vDash\varphi$.
for all $\varphi\in\Lambda$ . In particular, $x\vDash\varphi\Leftrightarrow y\vDash\varphi$.
for all $\varphi$ with $\boldsymbol{\varphi}\in D$ which implies
that$y\in\mathbf{x}_{D}.$
\end{proof}
\begin{rem}
When we assume a descriptor representative, we state so. Though modal
spaces for representative descriptor are of prime interest, for several
results the assumption is not necessary.
\end{rem}

\subsection{Metrics on Modal Spaces}

Finally, we obtain the family $\mathcal{D}_{(X,D)}$ of metrics on
the $D$-modal space of a set of pointed Kripke models $X$:
\begin{prop}
\label{prop:Metric}Let $D$ be an enumerated descriptor for the set
of pointed Kripke models $X$. Let $\nu:\boldsymbol{X}_{D}\times D\rightarrow\{0,1\}$
be a valuation given by $\nu(\boldsymbol{x},\boldsymbol{\varphi})=1$
iff $x\models\varphi$ for all $\boldsymbol{x}\in\boldsymbol{X}_{D},\boldsymbol{\varphi}\in D$.
Let $w:D\rightarrow\mathbb{R}_{>0}$ we a weight function. Then $d_{w}$
is a metric on $\boldsymbol{X}_{D}$.
\end{prop}

\begin{proof}
This follows immediately from Proposition \ref{prop:structure-metric}
as $\nu$ is well-defined, cf. Remark \ref{rem:Soundness}.
\end{proof}
\begin{cor}
For every $X$-descriptor $D$, $\mathcal{D}_{(X,D)}$ is a family
of metrics on $\boldsymbol{X}_{D}$.
\end{cor}

\subsection{Examples\label{subsec:Examples}}

In constructing a metric $d_{w}\in\mathcal{D}_{(X,D)}$ for some modal
space $\boldsymbol{X}_{D}$, two parameters must be fixed: The descriptor
and the weight function. Jointly, these two parameters allow much
freedom in picking a metric according to desired properties. In this
section, we provide three classes of examples: First of non-representative
descriptors, second of representative descriptors, and third of representative
descriptors on finite sets, where we by a general proposition prove
previous metrics on pointed Kripke models \cite{aucher2010generalizing,Caridroit2016}
special cases of our approach.

\subsubsection{Non-Representative Descriptors}
\begin{example}
\textbf{Hamming Distance on Partial Atom Valuation.} 

\noindent Let $\mathcal{L}$ be a modal language and $K$ and $X$
respectively the minimal normal modal logic and a set of pointed Kripke
models for $\mathcal{L}$. Let $p_{1},p_{2},...$ be an enumeration
of the atoms of $\mathcal{L}$. Pick as descriptor $D=\{\boldsymbol{p}_{1},...,\boldsymbol{p}_{n}\}\subseteq\boldsymbol{\mathcal{L}}_{K}$
and weight function $w$ given by $w(\boldsymbol{p}_{k})=1$ for all
$\boldsymbol{p}_{k}\in D$. Then $d_{w}$ is a metric on $\boldsymbol{X}_{D}$
cf. Prop. \ref{prop:Metric}. The metric space $(\boldsymbol{X}_{D},d_{w})$
is isomorphic to the metric space of strings of length $n$ under
the Hamming distance. In it, pointed Kripke models are compared only
by their valuation of the first $n$ atoms. The space and the underlying
metric reflects no modal structure.

If the set of atoms $\Phi$ of $\mathcal{L}$ is countably infinite,
then we \emph{cannot} assign all atoms equal weight: The sequence
$(w'(\boldsymbol{p}_{n}))_{n\in\mathbb{N}}$ would not give rise to
a convergent series, so $w'$ is not a weight function. Partitioning
$\Phi$ into cells $P_{1},P_{2},...$ with each $P_{k},k\in\mathbb{N}$
finite but arbitrarily large, and assigning $w''(\boldsymbol{p})=a_{k}$
for all $\boldsymbol{p}\in P_{k}$ with $a_{k}$ the $k$-th term
of some convergent series does, however, give rise to a weight function.
\end{example}

\begin{example}
\textbf{World Views and Situation Similarity.} 

\noindent Consider an agent, $a$, who cares only about her beliefs
about some of atom $p$ and her beliefs about the beliefs of another
agent, $b$, about the same. Working in a doxastic $KD45$ logic with
operators $B_{a}$ and $B_{b}$, agent $a$'s world view may be described
by $D=\{\boldsymbol{B_{a}\varphi},\boldsymbol{B_{a}\neg\varphi},\boldsymbol{\neg B_{a}\varphi\wedge\neg B_{a}\neg\varphi}\}$
with $\varphi\in\{p,B_{b}p,B_{b}\neg p\}$. Similarities in situations
(pointed Kripke models) from the viewpoint of $a$ may then be represented
by using weight functions and their distances. E.g.: If $a$ cares
equally much about her own and $b$'s beliefs, every element of $D$
may be given weight; If she cares less about $b$'s beliefs, $D$
may be suitably partitioned and weighted; Etc.
\end{example}

\subsubsection{Representative Descriptors}
\begin{example}
\textbf{Degrees of Bisimilarity.}\label{ex:n-bisim.metric}

\noindent  Contrary to the logico-syntactic approach to metric construction,
a natural semantic approach rests on \textbf{bisimulation}. In particular,
the notion of $n$\textbf{-bisimularity} may be used to define a semantically
based metric on quotient spaces of pointed Kripke models where degrees
of bisimilarity translate to closeness in space\textemdash the more
bisimilar, the closer:
\end{example}

Let $X$ be a set of pointed Kripke models for which modal equivalence
and bisimilarity coincide\footnote{That all models in $X$ are \textbf{image-finite} is a sufficient
condition, cf. the Hennessy-Milner Theorem. See e.g. \cite{BlueModalLogic2001}
or \cite{ModelTheoryModalLogic}.} and let $\leftrightarroweq_{n}$ relate $x,y\in X$ iff $x$ and
$y$ are $n$-bisimilar. Then

\setstretch{.8}
\begin{equation}
d_{B}(\boldsymbol{x},\boldsymbol{y})=\begin{cases}
0 & \text{ if }x\leftrightarroweq_{n}y\text{ for all }n\\
\frac{1}{n} & \text{ if }n\text{ is the least intenger such that }x\not\leftrightarroweq_{n}y
\end{cases}\label{eq:n-bisim}
\end{equation}
\setstretch{1.25}is a metric on $\boldsymbol{X}_{\boldsymbol{\mathcal{L}}_{K}}$.\footnote{The metric is inspired by \cite{Goranko2004}, defining a distance
between theories of first-order logic using quantifier depth, to which
we return in Section \ref{sec:n-bisim-topology}. Also aiming at a
bisimulation-based metric is the ``$n$-Bisimulation-based Distance''
of \cite{Caridroit2016}, which yields a pseudo-metric on sets of
finite, pointed Kripke models (see also Sec. \ref{subsec:Metrics-on-Finite}
below).} We refer to $d_{B}$ as the \textbf{$n$-bisimulation metric}.

For $X$ and $\mathcal{L}$ based on a finite signature, we have $d_{B}\in\mathcal{D}_{(X,D)}$,
i.e. the $n$-bisimulation metric is contained in the family introduced:
Note that each model in $X$ has a \textbf{characteristic formula}
up to $n$-bisimulation. I.e., for each $x\in X$, there exists a
$\varphi_{x,n}\in\mathcal{L}$ such that for all $y\in X,$ $y\models\varphi_{x,n}$
iff $x\leftrightarroweq_{n}y$, cf. \cite{moss2007finite,ModelTheoryModalLogic}.
Given that both $\Phi$ and $\mathcal{I}$ are finite, so is, for
each $n$, the set $D_{n}=\{\boldsymbol{\varphi_{x,n}}:x\in X\}\subseteq\boldsymbol{\mathcal{L}}_{K}$
with $K$ the minimal normal modal logic. Pick the set of descriptors
to be $D=\bigcup_{n\in\mathbb{N}}D_{n}$. Then $D$ is $K$-representative,
so $\boldsymbol{X}_{D}$ is identical to $\boldsymbol{X}_{\boldsymbol{\mathcal{L}}_{K}}$,
cf. Lemma \ref{lem:Repr.Descriptor}.

Let the weight function $b$ be given by
\[
b(\boldsymbol{\varphi})=\frac{1}{2}\left(\frac{1}{n}-\frac{1}{n+1}\right)\text{ for }\boldsymbol{\varphi}\in D_{n}.
\]

\noindent Hence $d_{b}$, defined by
\[
d_{b}(\boldsymbol{x},\boldsymbol{y})=\sum_{k=0}^{\infty}b(\boldsymbol{\varphi}_{k})\cdot d_{k}(\boldsymbol{x},\boldsymbol{y}),
\]
is a metric on $\boldsymbol{X}_{\boldsymbol{\mathcal{L}}_{K}}$ cf.
\ref{prop:Metric}. As models $x$ and $y$ will, for all $n$, either
agree on all members of $D_{n}$ or disagree on exactly 2 (namely
$\varphi_{n,x}$ and $\varphi_{n,y}$) and as, for all $k\leq n$,
$y\models\varphi_{n,x}$ implies $y\models\varphi_{k,x}$, and for
all $k\geq n$, $y\not\models\varphi_{n,x}$ implies $y\not\models\varphi_{k,x}$,
we obtain that

\setstretch{.8}
\[
d_{b}(\boldsymbol{x},\boldsymbol{y})=\begin{cases}
0 & \text{ if }x\leftrightarroweq_{n}y\text{ for all }n\\
\sum_{k=n}^{\infty}2\cdot\frac{1}{2}\left(\frac{1}{k}-\frac{1}{k+1}\right)=\frac{1}{n} & \text{ if }n\text{ is the least intenger such that }x\not\leftrightarroweq_{n}y
\end{cases}
\]
\setstretch{1.25} which is exactly $d_{B}$.
\begin{rem}
The construction given for encoding of the $n$-bisimulation metric
only works when the set of atoms and number of modalities are finite:
No metric in $\mathcal{D}_{(X,\boldsymbol{\mathcal{L}}_{K})}$ is
equivalent with the $n$-bisimulation metric in the case of infinitely
many atoms, cf. Section \ref{sec:n-bisim-topology}.
\end{rem}

\begin{example}
\textbf{Close to Home, Close to Heart.}\label{ex:org.metric}

\noindent The distances $d_{B}$ and $d_{b}$ do not reflect all differences
between models. For example, if two models are not $n$-bisimilar
due only to atomic disagreement $n$ steps from the designated state,
then it does not matter on how many atoms or how many worlds at distance
$n$ they disagree: Their distance will be $\frac{1}{n}$ in all cases.
Likewise, no differences they exhibit beyond the $n$th step will
influence their distance: Only the first difference matters.
\end{example}

In $\mathcal{D}_{(X,\boldsymbol{\mathcal{L}}_{K})}$, we find a metric
which retains the feature of $d_{b}$ that differences further from
the designated state weighs less than differences closer, but which
assigns a positive weight to every modal proposition. In a slogan:
\begin{quote}
\emph{All and only modally expressible difference matters, but the
further you have to go to find it, the less it matters}.
\end{quote}
On a set of finite atom models $X$, a metric that lives up to the
slogan may be defined as follows:

Take the descriptor to be $\boldsymbol{\mathcal{L}}_{K}$. Let $\{D_{n}\}_{n\in\mathbb{N}}$
be a partition of $D$ by \textbf{shallowest modal depth}: For $n\in\mathbb{N}$,
let $D_{n}$ contain the $K$-propositions $\boldsymbol{\varphi}$
for which the the shallowest $K$-representative $\chi\in\boldsymbol{\varphi}$
have modal depth $n$. I.e., with $md(\varphi)$ the modal depth of
$\varphi$,
\[
D_{n}=\{\boldsymbol{\varphi}\in D:\exists\chi\in\boldsymbol{\varphi},(md(\chi)=n)\mbox{ and }\forall\psi\in\boldsymbol{\varphi},(md(\psi)\geq n)\}.
\]
Define a weight function $c$ by
\[
c(\boldsymbol{\varphi})=\frac{1}{|D_{n}|}\frac{1}{\prod_{k<n}|D_{k}|}\frac{1}{2^{n}}\text{ for }\boldsymbol{\varphi}\in D_{n}.
\]
Then $d_{c}$ is a metric on $\boldsymbol{X}_{\boldsymbol{\mathcal{L}}_{K}}$.

The first term ensures that disagreement on any formula in $D_{n}$
contributes $\frac{1}{\prod_{k<n}|D_{k}|}\frac{1}{2^{n}}$ to the
distance between models. The second term ensures that the summed weight
of all formulas in $D_{j}$ for $j>n$ is less than or equal to the
weight of any $D_{n}$ formula, even when $|D_{j}|>|D_{n}|$. The
third term ensures that the summed weights will not be equal: One
disagreement on a single formula of modal depth $n$ adds more to
the distance between two models than do disagreement on all formulas
of modal depth $n+1$ and above. Formally, for all $n$,
\begin{equation}
\frac{1}{2^{n}}\frac{1}{|D_{n}|}\frac{1}{\prod_{k<n}|D_{k}|}>\sum_{m=n+1}^{\infty}\frac{1}{2^{m}}\frac{|D_{m}|}{|D_{m}|}\frac{1}{\prod_{k<n}|D_{k}|}.\label{eq:less}
\end{equation}
Given this features, the metric $d_{\vec{c}}$ captures both aspects
the slogan:
\begin{enumerate}
\item Given that every cell in $\boldsymbol{\mathcal{L}}_{K}$ is given
positive weight, and that only disagreement on these cells contribute
to the distance between model, all and only modally expressible differences
matter.
\item That further distance from the designated world should imply less
importance of difference is captured as Eq. (\ref{eq:less}) implies
that for any $x,y,z\in X$, if $x$ and $y$ are not $n$-modally
equivalent but $x$ and $z$ are, then $d_{c}(\boldsymbol{x},\boldsymbol{y})>d_{c}(\boldsymbol{x},\boldsymbol{z}).$
\end{enumerate}

\subsubsection{Metrics on Finite Sets \label{subsec:Metrics-on-Finite}}

As a last example, consider the case where $X$ and $\Lambda$ are
such that $\ms$ is of finite cardinality. This may happen e.g. in
a language with a single operator and finite atoms under $S5$ equivalence,
or if $X$ itself is finite, as is explicitly assumed in \cite{Caridroit2016}
when Cardroit \emph{et. al} define their $6$ distances between pointed
Kripke models. In this setting, for\textit{ any} metric $d$ on $\ms$
there is an equivalent metric $d_{b}\in\mathcal{D}_{(X,D)}$ such
that the spaces $(\ms,d)$ and $(\ms,d_{b})$ are quasi-isometric
to each other.
\begin{prop}
\label{prop:replem} Let $(\ms,d)$ be a finite metric space. Then
there exists a descriptor $D\subseteq\boldsymbol{\mathcal{L}}_{\Lambda}$
, a metric $d_{w}\in\mathcal{D}_{(X,D)}$ and some $c\geq0$ such
that $d_{w}(\boldsymbol{x}_{D},\mathbf{y}_{D})=d(\boldsymbol{x}_{\mathcal{L}_{\Lambda}},\boldsymbol{y}_{\mathcal{L}_{\Lambda}})+c$
for all $\boldsymbol{x}\neq\boldsymbol{y}\in\ms$.  In particular,
$(\boldsymbol{X}_{D},d_{w})$ are $(\ms,d)$ quasi-isometric to each
other. 
\end{prop}

\begin{proof}
Since $\ms$ is finite, there is a $\varphi_{\boldsymbol{x}}$ for
each $\boldsymbol{x}\in\ms$ such that for all $y\in X$, if $y\models\varphi_{\boldsymbol{x}}$,
then $y\in\boldsymbol{x}$. Moreover, let $\varphi_{\{\boldsymbol{x},\boldsymbol{y}\}}$
denote the formula $\varphi_{\boldsymbol{x}}\vee\varphi_{\boldsymbol{y}}$
which holds true in $\boldsymbol{x}\in\ms$ iff $\boldsymbol{z}=\boldsymbol{x}$
or $\boldsymbol{z}=\boldsymbol{y}$. Let $D=\{\boldsymbol{\varphi}_{\boldsymbol{x}}\colon x\in X\}\cup\{\boldsymbol{\varphi_{\{\boldsymbol{x},\boldsymbol{y}\}}}\colon x,y\in X\}$.
It follows that $\boldsymbol{X}_{D}=\ms$.

Next, partition the finite set $\ms\times\ms$ according to the metric
$d$: Let $S_{1},...,S_{k}$ be the unique partition of $\ms\times\ms$
that satisfies, for all $i,j\leq k$
\begin{enumerate}
\item If $(\boldsymbol{x},\boldsymbol{x'})\in S_{i}$ and $(\boldsymbol{y},\boldsymbol{y'})\in S_{i}$,
then $d(\boldsymbol{x},\boldsymbol{x'})=d(\boldsymbol{y},\boldsymbol{y'})$,
and
\item If $(\boldsymbol{x},\boldsymbol{x'})\in S_{i}$ and $(\boldsymbol{y},\boldsymbol{y'})\in S_{j}$
for $i<j$, then $d(\boldsymbol{x},\boldsymbol{x'})<d(\boldsymbol{y},\boldsymbol{y'})$.
\end{enumerate}
For $i\leq k$, let $b_{i}$ denote $d(x,y)$ for any $(x,y)\in S_{i}$.
Define a weight function $w:D\rightarrow\mathbb{R}_{>0}$ by 
\begin{align*}
w(\boldsymbol{\varphi_{x}}) & =\sum_{i=1}^{k}\sum_{\substack{(y,z)\in S_{i}\\
x\not\neq y,z
}
}\frac{1+b_{k}-b_{i}}{4}\\
w(\boldsymbol{\varphi_{\{x,y\}}}) & =2\cdot\frac{1+b_{k}-b_{i}}{4}\text{ for the }i\text{ with }(x,y)\in S_{i}
\end{align*}
Note that by symmetry, $(x,y)\in S_{i}$ implies $(y,x)\in S_{i}$,
thus $w(\varphi_{\{x,y\}})$ is well-defined. We get for each $x$
that 
\begin{align*}
w(\varphi_{x})+\sum_{y\neq x}w(\varphi_{\{x,y\}})=\sum_{i=1}^{k}\sum_{\substack{(y,z)\in S_{i}\\
x\not\in\{y,z\}
}
}\frac{1+b_{k}-b_{i}}{4}+\sum_{i=1}^{k}\sum_{\substack{(y,z)\in S_{i}\\
x\in\{y,z\}
}
}\frac{1+b_{k}-b_{i}}{4}=\sum_{i=1}^{k}\sum_{(y,z)\in S_{i}}\frac{1+b_{k}-b_{i}}{4}
\end{align*}
For simplicity, we denote the rightmost term $\sum_{i=1}^{k}\sum_{(y,z)\in S_{i}}\frac{1+b_{k}-b_{i}}{4}$
of the previous equation by $a$. Next, note that two models $x$
and $y$ differ on exactly the formulas $\varphi_{x},\varphi_{y}$
and all $\varphi_{\{x,z\}}$ and $\varphi_{\{y,z\}}$ for $z\neq x,y$.
In particular, we have that 
\begin{align*}
d_{w}(x,y)= & w(\varphi_{x})+w(\varphi_{y})+\sum_{z\neq x,y}w(\varphi_{\{x,z\}})+\sum_{z\neq x,y}w(\varphi_{\{y,z\}})\\
= & w(\varphi_{x})+w(\varphi_{y})+\sum_{z\neq x}w(\varphi_{\{x,z\}})+\sum_{z\neq y}w(\varphi_{\{y,z\}})-2w(\varphi_{\{x,y\}})=2a-4\cdot\frac{1+b_{k}-b_{i}}{4}=2a+b_{i}-1-b_{k}
\end{align*}
where $i$ is such that $\{x,y\}\in S_{i}$. In particular, we get
that $d_{w}(x,y)-d_{w}(a,b)=b_{i}-b_{j}=d(x,y)-d(a,b)$ whenever $(x,y)\in S_{i}$
and $(a,b)\in S_{j}$.

\end{proof}

\section{Topological Properties\label{sec:Topo}}

Given a set of pointed Kripke models $X$ and a descriptor $D\subseteq\boldsymbol{\mathcal{L}}_{\Lambda}$
for $\Lambda$ a modal logic sound w.r.t. $X$, Proposition \ref{prop:Metric}
states that for any weight function $w$, $d_{w}$ is metric on the
modal space $\boldsymbol{X}_{D}$, the quotient of $X$ under $D$-equivalence.
Hence $(\boldsymbol{X}_{D},d_{w})$ is a \textbf{metric space}. Any
such metric space induces a \textbf{topological space} $(\boldsymbol{X}_{D},\mathcal{T}_{w})$
with a basis consisting of the open $\epsilon$-balls of $(\boldsymbol{X}_{D},d_{w})$:
I.e., the basis of the $d_{w}$\textbf{ metric topology} $\mathcal{T}_{w}$
on $\boldsymbol{X}_{D}$ is $\{B_{d_{w}}(\boldsymbol{x},\varepsilon):\boldsymbol{x}\in\boldsymbol{X}_{D}\}$
with $B_{d_{w}}(\boldsymbol{x},\varepsilon)=\{\boldsymbol{y}\in\boldsymbol{X}_{D}:d_{w}(\boldsymbol{x},\boldsymbol{y})<\varepsilon\}$.
In this section, we investigate the topological properties of such
spaces.

\subsection{Stone-like Topologies}

In fixing a descriptor $D$ for $X$, one also fixes the family of
metrics $\mathcal{D}_{(X,D)}$. The members of $\mathcal{D}_{(X,D)}$
vary in their metrical properties, as evident from e.g. comparing
Examples \ref{ex:n-bisim.metric} and \ref{ex:org.metric}. They are
however topologically equivalent. To show this, we must work with
the following generalization of the \emph{Stone topology}:
\begin{defn}
\label{def:Stone-like.Topology} Let $D$ be a descriptor for $X$.
Define the \textbf{Stone-like topology} on $\boldsymbol{X}_{D}$ to
be the topology $\mathcal{T}_{D}$ given by the subbasis of sets $\{\boldsymbol{x}\in\boldsymbol{X}_{D}\colon x\models\varphi\}$
and $\{\boldsymbol{x}\in\boldsymbol{X}_{D}\colon x\models\neg\varphi\}$
for $\boldsymbol{\varphi}\in D$.
\end{defn}

Note that, as $D$ need not be closed under conjunction, this subbasis
is, in general, not a \textbf{basis} of the topology. When $D\subseteq\boldsymbol{\mathcal{L}}_{\Lambda}$
is $\Lambda$-representative, $\boldsymbol{X}_{D}$ is identical to
$\ms$, and the Stone-like topology $\mathcal{T}_{D}$ on $\boldsymbol{X}_{D}$
is identical to the \textbf{Stone topology }on $\ms$ given by the
basis of sets $\{\boldsymbol{x}\in\ms\colon x\models\varphi\}$, $\boldsymbol{\varphi}\in\boldsymbol{\mathcal{L}}_{\Lambda}$.

We may now state the promised proposition:
\begin{prop}
\label{prop:same-topology} The metric topology $\mathcal{T}_{w}$
of any metric $d_{w}\in\mathcal{D}_{(X,D)}$ on $\boldsymbol{X}_{D}$
is the Stone-like topology $\mathcal{T}_{D}$.
\end{prop}

\begin{proof}
We recall that for topologies $\mathcal{T}$ and \textbf{$\mathcal{T}'$}
on some set $X$, if $\mathcal{T}'\text{\ensuremath{\subseteq}}\mathcal{T}$,
then $\mathcal{T}'$ is said to be \textbf{finer than} $\mathcal{T}$,
and that this is the case iff for each $x\in X$ and each basis element
$B\in\mathcal{\mathcal{T}}$ with $x\in B$, there exists a basis
element $B'\in\mathcal{\mathcal{T}}'$ with $x\in B'\subseteq B$,
cf. \cite[Lem. 13.3]{Munkres}. \medskip{}

\noindent 1) The topology $\mathcal{T}_{w}$ is finer than $\mathcal{T}_{D}$
($\mathcal{T}_{w}\subseteq\mathcal{T}_{D}$): It suffices to show
the claim for all elements of a subbasis of $\mathcal{T}_{D}$. Let
$\boldsymbol{x}\in\boldsymbol{X}_{D}$ and let $B_{D}$ be a subbasis
element of $\mathcal{T}_{D}$ which contains $\boldsymbol{x}$. Then
$B_{D}$ is of the form $\{\boldsymbol{y}\in\boldsymbol{X}_{D}:y\models\varphi\}$
or $\{\boldsymbol{y}\in\boldsymbol{X}_{D}:y\models\text{\ensuremath{\neg}}\varphi\}$
for some $\varphi\in\boldsymbol{\varphi}\in D$. Wlog we assume the
former. As $\boldsymbol{x}\in B_{D}$, $x\models\varphi$. In the
metric $d_{w}$, $\boldsymbol{\varphi}$ is assigned a strictly positive
weight $w(\boldsymbol{\varphi})$. The open ball $B(\boldsymbol{x},w(\boldsymbol{\varphi}))$
of radius $w(\boldsymbol{\varphi})$ around $\boldsymbol{x}$ is a
basis element of $\mathcal{T}_{w}$ and contains $\boldsymbol{x}$.
Moreover, $B(\boldsymbol{x},w(\boldsymbol{\varphi}))\subseteq B_{D}$:
Assume $\boldsymbol{y}\in B(\boldsymbol{x},w(\boldsymbol{\varphi}))$,
but $y\not\models\varphi$. Then $d_{w}(\boldsymbol{x},\boldsymbol{y})\geq w(\boldsymbol{\varphi})$.
But then $\boldsymbol{y}\not\in B(x,w(\boldsymbol{\varphi}))$, contrary
to assumption. We conclude that $\mathcal{T}_{w}$ is finer than $\mathcal{T}_{D}$.\medskip{}

\noindent 2) The topology $\mathcal{T}_{D}$ is finer than $\mathcal{T}_{w}$
($\mathcal{T}_{D}\subseteq\mathcal{T}_{w}$): Let $B$ be a basis
element of $\mathcal{T}_{w}$ which contains $\boldsymbol{x}$. As
$B$ is a basis element, it is of the form $B(\boldsymbol{y},\delta)$
for some $\delta>0$. Let $\epsilon=\delta-d_{w}(\boldsymbol{x},\boldsymbol{y})$.
Note that $\epsilon>0$. Let $\boldsymbol{\varphi}_{1},\boldsymbol{\varphi}_{2},...$
be an enumeration of $D$. Since $\sum_{i=0}^{\infty}w(\boldsymbol{\varphi}_{i})<\infty$
, there is some $n$ such that $\sum_{i=n}^{\infty}w(\boldsymbol{\varphi}_{i})<\epsilon$.
For $i<n$, pick some $\chi_{i}\in\boldsymbol{\varphi}_{i}$ if $x\vDash\varphi_{i}$
and some as $\chi_{i}$ with $\neg\chi_{i}\in\boldsymbol{\varphi}_{i}$
otherwise. Let $\chi=\bigwedge_{i<n}\chi_{i}$. By construction, all
$z$ with $z\vDash\chi$ agree with $x$ on the truth values of $\varphi_{1},\ldots,\varphi_{n-1}$
and thus $d_{w}(\boldsymbol{x},\boldsymbol{z})<\epsilon$. By the
triangular inequality, this implies $d_{w}(\boldsymbol{y,z})<\delta$
and hence $\{\boldsymbol{z}\colon z\vDash\varphi\}\subseteq B$. Furthermore,
since $\mathcal{T}_{D}$ is generated by $\{\boldsymbol{x}\in\boldsymbol{X}_{D}\colon x\models\varphi\}$
and $\{\boldsymbol{x}\in\boldsymbol{X}_{D}\colon x\not\models\varphi\}$
for $\varphi\in\boldsymbol{\varphi}\text{\ensuremath{\in}}D$, we
have $\{\boldsymbol{z}\colon z\vDash\varphi\}\in\mathcal{T}_{D}$
as desired.
\end{proof}

\subsection{Stone Spaces}

 The Stone topology is well-known, but typically defined on the set
of ultrafilters of a Boolean algebra, which it turns into a \textbf{Stone
space}: A \textbf{totally disconnected}, \textbf{compact}, \textbf{Hausdorff}
topological space.

When applying Stone-like topologies to modal spaces, Stone spaces
often result. That the resulting topological spaces are Hausdorff
follows as each Stone-like topology is metrizable, cf. the previous
section. We show that the Stone-like topology is also totally disconnected
and identify sufficient conditions for its compactness.
\begin{prop}
\label{prop:discon} For any $X$-descriptor $D$, the space $(\boldsymbol{X}_{D},\mathcal{T}_{D})$
is totally disconnected.

\begin{proof}
Let $\boldsymbol{x}\neq\boldsymbol{y}\in\boldsymbol{X}_{D}$. We must
find open sets $U,V$ with $\boldsymbol{x}\in U$ and $\boldsymbol{y}\in Y$
such that $U\cup V=\boldsymbol{X}_{D}$ and $U\cap V=\emptyset$.
Since $\boldsymbol{x}\neq\boldsymbol{y}$, there exists some $\boldsymbol{\varphi}\in D$
such that $x\models\varphi$ while $y\not\models\varphi$. The sets
$A=\{\boldsymbol{z}\in\boldsymbol{X}_{D}\colon z\models\varphi\}$
and $\overline{A}=\{\boldsymbol{z}'\in\boldsymbol{X}_{D}\colon z\models\neg\varphi\}$
are both open in the Stone-like topology, $A\cup\overline{A}=\boldsymbol{X}_{D}$
and $A\cap\overline{A}=\emptyset$. As $\boldsymbol{x}\in A$ and
$\boldsymbol{y}\in\overline{A}$, this shows that the space  $(\boldsymbol{X}_{D},\mathcal{T}_{D})$
is totally disconnected.
\end{proof}
\end{prop}

The space $(\boldsymbol{X}_{D},\mathcal{T}_{D})$, $D\subseteq\boldsymbol{\mathcal{L}}_{\Lambda}$,
is moreover compact when two requirements are satisfied: First, the
logic $\Lambda$ underlying $D$ must be \textbf{logically compact:
}An arbitrary set $A\subseteq\mathcal{L}$ of formulas is $\Lambda$-consistent
iff every finite subset of $A$ is also $\Lambda$-consistent. Many
modal logics are compact, including every \textbf{basic modal logic},
cf. e.g. \cite{BenthemBlackburn}, but not all are: Examples include
logics with a \emph{common knowledge operator }\cite[7.3]{Ditmarsch2008}
or with \emph{Kleene star}\textbf{ }as a \emph{PDL constructor }\cite[4.8]{BlueModalLogic2001}.
As the second requirement, we must assume the set $X$ sufficiently
rich in model diversity:
\begin{defn*}
Let $D\subseteq\boldsymbol{\mathcal{L}}_{\Lambda}$ be an $X$-descriptor.
Say that $X$ is \textbf{saturated }with respect to $D$ if for all
subsets $Y,Y'\subseteq D$ such that $B=\{\varphi\in\mathcal{L}:\boldsymbol{\varphi}\in Y\}\cup\{\text{\ensuremath{\neg}}\varphi\in\mathcal{L}:\boldsymbol{\varphi}\in Y'\}$
is $\Lambda$-consistent, there exists a model $x$ in $X$ such that
$x\models\psi$ for all $\psi\in B$.
\end{defn*}
Under these two requirements, we obtain the following:
\begin{prop}
\label{prop:compact} If $\Lambda$ is a compact and $X$ is saturated
with respect to $D\subseteq\boldsymbol{\mathcal{L}}_{\Lambda}$, then
the space $(\boldsymbol{X}_{D},\mathcal{T}_{D})$ is compact.
\end{prop}

\begin{proof}
\noindent Note that a basis of the topology $\mathcal{T}_{D}$ is given by the
family of all sets $\{\boldsymbol{x}\in\boldsymbol{X}_{D}\colon x\models\chi\}$,
where $\chi$ is of the form $\chi=\psi_{1}\wedge\ldots\wedge\psi_{n}$
for some $n$ such that for all $i\leq n$ either $\psi_{i}\in\boldsymbol{\varphi}_{j}\in D$
or $\neg\psi_{i}\in\boldsymbol{\varphi}_{j}\in D$ . To show that
$(\boldsymbol{X}_{D},\mathcal{T}_{D})$ is compact, it suffices to
show that every open cover consisting of basic open sets has a finite
subcover. Suppose that $\{\{\boldsymbol{x}\in\boldsymbol{X}_{D}\colon x\models\chi_{i}\}\colon i\in I\}$
is a cover of $\boldsymbol{X}$ but that contains no finite subcover.
This implies that every finite subset $\{\neg\chi_{i}\colon i\in I\}$
is consistent, i.e., the set $\{\neg\chi_{i}\chi_{i}\colon i\in I\}$
is finitely $\Lambda$-consistent. By the compactness of $\Lambda$,
$\{\neg\chi_{i}\chi_{i}\colon i\in I\}$ itself is thus $\Lambda$-consistent.
By saturation, there is an $x\in X$ such that $x\models\neg\chi_{i}$
for all $i\in I$. But then $\boldsymbol{x}$ cannot be in $\{\boldsymbol{x}\in\boldsymbol{X}_{D}\colon x\models\chi_{i}\}$
for any $i\in I$. This contradicts that $\{\{\boldsymbol{x}\in\boldsymbol{X}_{D}\colon x\models\chi_{i}\}\colon i\in I\}$
is a cover of $\boldsymbol{X}$.
\end{proof}
Propositions \ref{prop:discon} and \ref{prop:compact} jointly yields
the following:
\begin{cor}
Let $\Lambda$ be a compact modal logic sound and complete with respect
to the class of pointed Kripke models $\mathcal{C}$. Then $(\mathcal{C}_{\boldsymbol{\mathcal{L}}_{\Lambda}},\mathcal{T}_{\boldsymbol{\mathcal{L}}_{\Lambda}})$
is a Stone space.

\noindent \begin{proof}
\textit{\emph{The statement follows immediately the propositions of
this section when }}$\mathcal{C}_{\boldsymbol{\mathcal{L}}_{\Lambda}}$\textit{\emph{
is ensured to be a set using }}\textbf{\textit{\emph{Scott's trick}}}
\cite{ScottsTrick}.
\end{proof}
\end{cor}

\subsubsection{Compact Subspaces}

As the intersection of an arbitrary family of closed sets is itself
a closed set in any topology and as every closed subspace of a compact
space is compact (\cite[Thms 17.1, 26.2]{Munkres}), we obtain the
following, making use of the fact that $\{y\in X:y\models\varphi\}=X-\{y\in X:y\vDash\neg\varphi\}$
is closed for any $\varphi\in D$.
\begin{cor}
Let $A\subseteq D$ and let $Y=X\cap\{y\in X:y\models\varphi\text{ for all }\boldsymbol{\varphi}\in A\}$.
If $(\boldsymbol{X}_{D},\mathcal{T}_{D})$ is compact, then $\boldsymbol{Y}_{D}$
is compact under the subspace topology. 
\end{cor}

\noindent Moreover, the subspace topology when removing such $D$-definable
sets of models is again the Stone topology.

\subsection{Open, Closed and Clopen Sets in Stone-like Topologies}

In this section, we characterize the open, closed and clopen sets
of Stone-like topologies relative to the set of $\Lambda$-propositions.
With this, we hope to paint a logical picture of the structure of
Stone-like topologies, helpful in understanding closed subspaces and
limit points.

Given the modal space $\boldsymbol{X}_{D}$, $D\subseteq\boldsymbol{\mathcal{L}}_{\Lambda}$,
let $[\boldsymbol{\varphi}]_{D}=\{\boldsymbol{x}\in\boldsymbol{X}_{D}:\forall x\in\boldsymbol{x},x\models\varphi\}$
for each $\boldsymbol{\varphi}\in\boldsymbol{\mathcal{L}}_{\Lambda}$.
While this is well-defined for all $\boldsymbol{\varphi}\in\boldsymbol{\mathcal{L}}_{\Lambda}$,
there might be degenerate cased where $[\boldsymbol{\varphi}]_{D}\cup[\boldsymbol{\neg\varphi}]_{D}\not=\boldsymbol{X}_{D}$,
i.e. there may be some$\mathbf{x}_{D}\in\mathbf{X}_{D}$ such that
$\mathbf{x}\not\subseteq[\varphi]$, and$\mathbf{x}\not\subseteq[\neg\varphi].$If
$D$ is representative no such degerate cases occur, i.e. $[\boldsymbol{\varphi}]_{D}\cup[\boldsymbol{\neg\varphi}]_{D}=\boldsymbol{X}_{D}$
for all $\boldsymbol{\varphi}\in\boldsymbol{\mathcal{L}}_{\Lambda}$

By definition, the Stone-like topology $\mathcal{T}_{D}$ is generated
by the subbasis $\mathcal{S}_{D}=\{[\boldsymbol{\varphi}]_{D},[\boldsymbol{\neg\varphi}]_{D}\colon\boldsymbol{\varphi}\in D\}$.
All subbasis elements are clearly clopen: If $U$ is of the form $[\boldsymbol{\varphi}]_{D}$
for some $\boldsymbol{\varphi}\in D$, then the complement of $U$
is the set $[\boldsymbol{\neg\varphi}]_{D}$, which again is a subbasis
element. Hence both $[\boldsymbol{\varphi}]_{D}$ and $[\boldsymbol{\neg\varphi}]_{D}$
are clopen. As being clopen entails having empty \textbf{boundary},
the $\Lambda$-propositions $\boldsymbol{\varphi}$ and $\boldsymbol{\neg\varphi}$
are thus unambiguously reflected by the topology.
\begin{defn*}
Say that the Stone-like topology $\mathcal{T}_{D}$, $D\subseteq\ll$,
on the modal space $\boldsymbol{X}_{D}$ \textbf{reflects $\Lambda$}
if for every set $Y\subseteq\boldsymbol{X}_{D}$, $Y$ is clopen in
$\mathcal{T}_{D}$ iff $Y=[\boldsymbol{\varphi}]_{D}$ for some $\boldsymbol{\varphi}\in\boldsymbol{\mathcal{L}}_{\Lambda}$.
\end{defn*}
We immediately obtain the following: 
\begin{prop}
\noindent \label{prop:stoneclopen}For any modal space $\boldsymbol{X}_{D}$,
$D\subseteq\ll$, if $\Lambda$ is compact and $D$ is $\Lambda$-representative,
then $[\boldsymbol{\varphi}]_{D}$ is clopen in $\mathcal{T}_{D}$,
for every $\boldsymbol{\varphi}\in\ll$. If $\boldsymbol{X}_{D}$
is also saturated, then $\td$ reflects $\Lambda$.
\end{prop}

\begin{proof}
\noindent \noindent We start to show that under the assumptions, $[\boldsymbol{\varphi}]_{D}$
is clopen in $\mathcal{T}_{D}$, for every $\boldsymbol{\varphi}\in\ll$.
We first show the claim for the special case where $X$ is the set
of \textit{all} K-models that satisfy $\Lambda$. It suffices to show
that $\{x\in\boldsymbol{{X}}_{D}:x\models\varphi\}$ is open for $\varphi\in\mathcal{L}_{\Lambda}-D$.
Fix such $\varphi$. As $D$ is $\Lambda$-representative, $\boldsymbol{X}_{D}$
is identical to $\boldsymbol{X}_{\mathcal{L}_{\Lambda}}$, hence $[\varphi]:=\{\boldsymbol{x}\in X_{D}\ :\ x\models\varphi\}$
is well-defined. To see that it is open, assume $\boldsymbol{x}\in[\varphi]$.
We find an open set $U$ with $\boldsymbol{x}\in U\subseteq[\varphi]$:
Let $D_{x}=\{\boldsymbol{\psi}\in D\ \colon\ x\models\psi\}\cup\{\neg\psi\colon\boldsymbol{\psi}\in D\text{ and }x\models\neg\psi\}$.
The set $D_{x}\cup\{\varphi\}$ is $\Lambda$-consistent. Moreover,
as X is saturated with respect to $D$, the set $D_{x}\cup\{\neg\varphi\}$
is $\Lambda$-inconsistent. By compactness, a finite subset $F$ of
$D_{x}\cup\{\neg\varphi\}$ is inconsistent. As $D_{x}$ is consistent,
$F$ contains $\varphi$ and some formulas $\psi_{1},\ldots,\psi_{n}\in D_{x}$.
As $F$ is inconsistent, we get that $\psi_{1}\wedge\ldots\wedge\psi_{n}\rightarrow\varphi$
is a theorem of $\Lambda$. On a semantic level, this implies that
$\bigcap_{i\leq n}[\psi_{i}]\subseteq[\varphi]$. As each $[\psi_{i}]$
is open, $\bigcap_{i\leq n}[\psi_{i}]\subseteq[\varphi]$ is an open
neighborhood of \textbf{$\boldsymbol{x}$} contained in $[\varphi]$.
Next, we proof the general case. Let $X$ be any set of $\Lambda$-models
and let $Y$ be the set of \textit{all} K-models that satisfy $\Lambda$.
Then the function $f:\boldsymbol{{X}_{D}}\rightarrow\boldsymbol{{Y}}_{D}$
that sends $\boldsymbol{{x}}\in\boldsymbol{{X}_{D}}$ to the unique
$\boldsymbol{{x}}\in\boldsymbol{{Y}}_{D}$ with $x\vDash\varphi\Leftrightarrow y\vDash\varphi$
for all $\varphi\in\mathcal{L}$ is a continuous map from $(\boldsymbol{X}_{D},\mathcal{T}_{D})$
to $(\boldsymbol{Y}_{D},\mathcal{T}_{D})$. with $f^{-1}\left(\{y\in\boldsymbol{{Y}}_{D}:y\models\varphi\}\right)=\{x\in\boldsymbol{{X}}_{D}:x\models\varphi\}$.
By the first part, $\{y\in\boldsymbol{{Y}}_{D}:y\models\varphi\}$
is clopen. As the continuous pre-image of clopen sets is clopen, this
shows that $\{x\in\boldsymbol{{X}}_{D}:x\models\varphi\}$ is clopen.

Now we show that if $\boldsymbol{X}_{D}$ is also saturated, then
$\td$ reflects $\Lambda$. It suffices to show that if O\textbackslash{}subseteq
X\_D is clopen, then O is of the form$[\varphi]_{D}$ for some $\varphi\in\mathcal{L}$.
So assume $O$ is clopen. As $O$ and its complement $\overline{O}$
are open, there are formulas $\psi_{i},\chi_{i}$ for $i\in\mathbb{N}$
such that $O=\bigcup_{i<\mathbb{N}}[\psi_{i}]_{D}$ and $\overline{O}=\bigcup_{i<\mathbb{N}}[\chi_{i}]_{D}$.
The latter is equivalent to $O=\bigcap_{i<\mathbb{N}}[\neg\chi_{i}]_{D}$.
In particular, we have for all $k$ that $\bigcup_{i<k}[\psi_{i}]_{D}\subseteq O\subseteq\bigcap_{i<k}[\neg\chi_{i}]_{D}$.
We are interested in the sets $\bigcap_{i<k}[\neg\chi_{i}]_{D}-\bigcup_{i<k}[\psi_{i}]_{D}$
for $k\in\mathbb{N}$. To this end, let $\rho_{i}=\bigwedge_{i<k}\neg\chi_{i}\wedge\neg\left(\bigvee_{i<k}\psi_{i}\right)$,
hence $[\rho_{i}]_{D}=\bigcap_{i<k}[\neg\chi_{i}]_{D}-\bigcup_{i<k}[\psi_{i}]_{D}$.
Note that $\vdash\rho_{i+1}\rightarrow\rho_{i}$ and that 
\[
\bigcap_{k\in\mathbb{N}}[\rho_{k}]_{D}=\bigcap_{k\in\mathbb{N}}\left(\bigcap_{i<k}[\neg\chi_{i}]_{D}-\bigcup_{i<k}[\psi_{i}]_{D}\right)=\bigcap_{i\in\mathbb{N}}[\neg\chi_{i}]_{D}-\bigcup_{i\in\mathbb{N}}[\psi_{i}]_{D}=X-X=\emptyset
\]
As $X$ is saturated with respect to $D$, this implies that the set
$\{\rho_{i}\colon i\in\mathbb{N}\}$ is inconsistent. By compactness
of $\Lambda$, there is a finite subset $S\subseteq\{\rho_{i}\colon i\in\mathbb{N}\}$
that is already inconsistent. Let $i_{0}$ be the largest index occurring
in this subset. As $\rho_{i_{0}}\rightarrow\rho_{j}$ for every $j<i_{0}$we
have that $\{\rho_{i_{0}}\}$ is also inconsistent; hence $\emptyset=[\rho_{i_{0}}]_{D}$.
By saturation this implies that $\bigcup_{i\leq i_{0}}[\psi_{i}]_{D}=O=\bigcap_{i\leq i_{0}}[\neg\chi_{i}]_{D}$.
In particular, $O=[\bigvee_{i\leq i_{0}}\psi_{i}]_{D}$ which is,
what we had to show.  
\end{proof}
Compactness is essential to the characterization of clopen sets in
terms of $\Lambda$-proposition extensions of Proposition \ref{prop:stoneclopen}.
Without the assumption of compactness, the clopen sets of Stone topologies
do not reflect the underlying logic:
\begin{prop}
\label{prop:stonenotclopen}Let $\boldsymbol{X}_{D}$ be saturated
and $D\subseteq\ll$ $\Lambda$-representative, but $\Lambda$ not
compact. Then there exists a set $U$ clopen in $\mathcal{T}_{D}$
not of the form $[\boldsymbol{\varphi}]_{D}$, for any $\boldsymbol{\varphi}\in\ll$.
\end{prop}

\begin{proof}
In this proof, we omit the subscript from $[\boldsymbol{\varphi}]_{D}\subseteq\md=\ms$.

As $\Lambda$ is not compact, we can pick a set of formulas $\chi_{i},i\in\mathbb{N}$
such that $\{\chi_{i}\colon i\in\mathbb{N}\}$ is inconsistent, yet
every finite subset of $S$ is consistent. For simplicity of notation,
define $\varphi_{i}:=\neg\chi_{i}$ As $\md$ is saturated, $\{[\boldsymbol{\varphi_{i}}]\}_{i\in\mathbb{N}}$
is an open cover of $\md$ that does not contain a finite subcover.
Let $\rho_{i}$ be the formula $\varphi_{i}\wedge\bigwedge_{k<i}\neg\varphi_{k}$.
In particular we have that $i)$ $[\boldsymbol{\rho_{i}}]\cap[\boldsymbol{\rho_{j}}]=\emptyset$
for all $i\neq j$ and $ii)$ $\bigcup_{i\in\mathbb{N}}[\boldsymbol{\varphi_{i}}]=\bigcup_{i\in\mathbb{N}}[\boldsymbol{\rho_{i}}]=\md$.
I.e., $\{[\boldsymbol{\rho_{i}}]\}_{i\in\mathbb{N}}$ is a cover of
$\md$. We further have that $[\boldsymbol{\rho_{i}}]\subseteq[\boldsymbol{\varphi_{i}}]$;
hence $\{[\boldsymbol{\rho_{i}}]\}_{i\in\mathbb{N}}$ cannot contain
a finite subcover $\{[\boldsymbol{\rho_{i}}]\}_{i\in I}$ of $\md$,
as the respective $\{[\boldsymbol{\varphi_{i}}]\}_{i\in I}$ would
form a finite cover. Wlog we assume that all $[\boldsymbol{\rho_{i}}]$
are non-empty. For all $S\subseteq\mathbb{N}$, the set $U_{S}=\bigcup_{i\in S}[\boldsymbol{\rho_{i}}]$
is open. As all $[\boldsymbol{\rho_{i}}]$ are mutually disjoint,
the complement of $U_{S}$ is $\bigcup_{i\not\in S}[\boldsymbol{\rho_{i}}]$
which is also open; hence $U_{S}$ is clopen. Again as all $[\boldsymbol{\rho_{i}}]$
are mutually disjoint and non-empty, we have that $U_{S}\neq U_{S'}$
whenever $S\neq S'$. Hence, $\{U_{S}\colon S\subseteq\mathcal{\mathbb{N}}\}$
is an uncountable family of clopen sets. As $\ll$ is countable, there
must be some element of $\{U_{S}\colon S\subseteq\mathcal{\mathbb{N}}\}$
which is not of the form $[\boldsymbol{\varphi}]$ for any $\boldsymbol{\varphi}\in\ll$.
\end{proof}

\subsection{Relations to the $n$-Bisimulation Topology\label{sec:n-bisim-topology}}

In Example \ref{ex:n-bisim.metric}, we showed that $\mathcal{D}_{(X,\boldsymbol{\mathcal{L}}_{\Lambda})}$
includes the semantically based $n$-bisimulation metric $d_{B}$
for modal languages with finite signature. The metric topology induced
by the $n$-bisimulation metric is referred to as the \textbf{$n$-bisimulation
topology}, $\mathcal{T}_{B}$. A basis for this topology is given
by all subsets of $\ms$ of the form
\[
B_{\boldsymbol{x}n}=\{\boldsymbol{y}\in\ms\colon y\leftrightarroweq_{n}x\}.
\]

By Proposition \ref{prop:same-topology} and Example \ref{ex:n-bisim.metric},
we obtain the following:
\begin{cor}
\label{cor:nBisim.def.finite.sig}If $\mathcal{L}$ has finite signature,
then the $n$-bisimulation topology $\mathcal{T}_{B}$ is the Stone(-like)
topology $\mathcal{T}_{\boldsymbol{\mathcal{L}}_{\Lambda}}$. 
\end{cor}

This is not the case in general:
\begin{prop}
\label{prop:nBisim-is-not-Stone}If $\mathcal{L}$ is based on an
infinite set of atoms, then the $n$-bisimulation topology $\mathcal{T}_{B}$
is strictly finer than the Stone(-like) topology $\mathcal{T}_{\boldsymbol{\mathcal{L}}_{\Lambda}}$
on $\ms$.
\end{prop}

\begin{proof}
\noindent To see that the Stone(-like) topology is not as fine as the $n$-bisimulation
topology, consider the basis element $B_{\boldsymbol{x}0}$, containing
exactly the elements $\boldsymbol{y}$ such that $y$ and $x$ are
$0$-bisimilar, i.e., share atomic valuation. Clearly, $\boldsymbol{x}\in B_{\boldsymbol{x}0}$.
There is no formula $\varphi$ for which the Stone basis element $B=\{\boldsymbol{z}\in\boldsymbol{X}\colon z\models\varphi\}$
contains $\boldsymbol{x}$ and is contained in $B_{\boldsymbol{x}0}$:
This would require that $\varphi$ implied every atom or its negation,
requiring the strength of an infinitary conjunction.

For the inclusion of the Stone(-like) topology in the $n$-bisimulation
topology, consider any $\varphi\in\mathcal{L}$ and the corresponding
Stone basis element $B=\{\boldsymbol{y}\in\boldsymbol{X}\colon y\models\varphi\}$.
Assume $\boldsymbol{x}\in B$. Let the modal depth of $\varphi$ be
$n$. Then for every $\boldsymbol{z}\in B_{\boldsymbol{x}n}$, $z\models\varphi$.
Hence $\boldsymbol{x}\in B_{\boldsymbol{x}n}\subseteq B$.
\end{proof}
The discrepancy in induced topologies results as the $n$-bisimulation
metric, in the infinite case, introduces distinctions not made by
the logic: In the infinite case, there does not exist a characteristic
formula $\varphi_{x,n}$ satisfied only by models $n$-bisimilar with
$x$.

\paragraph*{Non-compactness.}

Even if $\ms$ is compact in the Stone(-like) topology, it need not
be compact in the $n$-bisimulation topology: Let $\mathcal{L}$ be
based on an infinite set of atoms $\Phi$ and $X$ a set of pointed
models saturated with respect to $\boldsymbol{\mathcal{L}}_{\Lambda}$.
Then $\ms$ is compact in the Stone(-like) topology. It is not compact
in the $n$-bisimulation topology: $\{B_{\boldsymbol{x}0}\colon x\in X\}$
is an open cover of $\ms$ which contains no finite subcover.

\paragraph*{Relations to Goranko (2004).}

Corollary \ref{cor:nBisim.def.finite.sig} and Proposition \ref{prop:nBisim-is-not-Stone}
jointly relate our metrics to the metric introduced by Valentin Goranko
in \cite{Goranko2004} on first-order theories. The straight-forward
alteration of that metric to suit a modal space $\boldsymbol{X}_{\boldsymbol{\mathcal{L}}_{\Lambda}}$
is
\[
d_{g}(\boldsymbol{x},\boldsymbol{y})=\begin{cases}
0 & \text{ if }\boldsymbol{x}=\boldsymbol{y}\\
\frac{1}{n+1} & \text{ if }n\text{ is the least intenger such that }n(\boldsymbol{x})\not=n(\boldsymbol{y})
\end{cases}
\]
where $n(\boldsymbol{x})$ is the set of formulas of modal depth $n$
satisfied by $x\in\boldsymbol{x}$. 

The induced topology of this metric is exactly the $n$-bisimulation
topology. Hence, for languages with finite signature, every metric
in our family $\mathcal{D}_{(X,\boldsymbol{\mathcal{L}}_{\Lambda})}$
induces the same topology as $d_{g}$, but the induced topologies
differ on languages with infinitely many atoms. 

Goranko notes in \cite{Goranko2004} that his topological approach
to prove relative completeness may, given a bit of work, be applied
in a modal logical setting.\footnote{See \S 6, especially the final paragraph.}
Replacing, in our approach, the modal space $\boldsymbol{X}_{\boldsymbol{\mathcal{L}}_{\Lambda}}$
with the quotient space of $X$ under bisimulation would, we venture,
supply the stepping stone. We omit a detour into the details in favor
of working with Stone-like topologies.

\section{Maps and Model Transformations\label{sec:Maps}}

In dynamic epistemic logic, dynamics are introduced by transitioning
between pointed Kripke models from some set $X$ using a possibly
partial map $f:X\longrightarrow X$ often referred to as a \textbf{model
transformer}. Many model transformers have been suggested in the literature,
the most well-known being \textbf{truthful public announcement} \cite{Plaza1989},
$!\varphi$, which maps $x$ to $x_{|\varphi}$, restriction of $x$
to the truth set of $\varphi$. Truthful public announcements are
a special case of a rich class of model transformers definable through
a particular graph product, \emph{product update}, of pointed Kripke
models with \emph{action models}. Due to their generality, popularity
and wide applicability, we focus on a general class of maps on modal
spaces induced by action models applied using product update.

An especially general version of action models is \emph{multi-pointed}
action models with \emph{postconditions}. Postconditions allow action
states in an action model to change the valuation of atoms \cite{Benthem2006_com-change,Ditmarsch_Kooi_ontic},
thereby also allowing the representation of information dynamics concerning
situations that are not factually static. Permitting multiple points
allows the actual action states executed to depend on the pointed
Kripke model to be transformed, thus generalizing single-pointed action
models. Multi-pointed action models are also referred to as \emph{epistemic
programs} in \cite{BaltagMoss2004}, and allow encodings akin to \emph{knowledge-based
programs} \cite{Fagin_etal_1995} of interpreted systems, cf. \cite{Rendsvig-DS-DEL-2015}.
Allowing for multiple points renders the class of action models \textbf{Turing
complete} \cite{BolanderBirkegaard2011}, even when not allowing for
atomic valuation change using postconditions \cite{KleinRendsvigTuring}.

\subsection{Action Models and Product Update}

A \textbf{multi-pointed action model} is a tuple $\Sigma{\scriptstyle \Gamma}=(\llbracket\Sigma\rrbracket,\mathsf{R},pre,post,\Gamma)$
where $\left\llbracket \Sigma\right\rrbracket $ is a countable, non-empty
set of \textbf{actions}. The map $\mathsf{R}:\mathcal{I}\rightarrow\mathcal{P}(\left\llbracket \Sigma\right\rrbracket \times\left\llbracket \Sigma\right\rrbracket )$
assigns an \textbf{accessibility relation} $\mathsf{R}_{i}$ on $\llbracket\Sigma\rrbracket$
to each agent $i\in\mathcal{I}$. The map ${pre:\left\llbracket \Sigma\right\rrbracket \rightarrow\mathcal{L}}$
assigns to each action a \textbf{precondition}, and the map ${post:\left\llbracket \Sigma\right\rrbracket \rightarrow\mathcal{L}}$
assigns to each action a \textbf{postcondition},\footnote{The precondition of $\sigma$ specify the conditions under which $\sigma$
is executable, while its postcondition may dictate the posterior values
of a finite, possibly empty, set of atoms.} which must be $\top$ or a conjunctive clause\footnote{I.e. a conjuction of literals, where a literal is an atom or a negated
atom.} over $\Phi$. Finally, $\emptyset\not=\Gamma\subseteq\left\llbracket \Sigma\right\rrbracket $
is the set of \textbf{designated actions}.

To obtain well-behaved total maps on a modal spaces, we must invoke
a set of mild, but non-standard, requirements: Let $X$ be a set of
pointed Kripke models. Call $\Sigma{\scriptstyle \Gamma}$ \textbf{precondition
finite} if the set $\{\boldsymbol{pre(\sigma)}\in\boldsymbol{\mathcal{L}}_{\Lambda}\colon\sigma\in\left\llbracket \Sigma\right\rrbracket \}$
is finite. This is needed for our proof of continuity. Call $\Sigma{\scriptstyle \Gamma}$
\textbf{exhaustive over }$X$ if for all $x\in X$, there is a $\sigma\in\Gamma$
such that $x\vDash pre(\sigma)$. This conditions ensures that the
action model $\Sigma{\scriptstyle \Gamma}$ is universally applicable
on $X$. Finally, call $\Sigma{\scriptstyle \Gamma}$ \textbf{deterministic
over }$X$ if $X\vDash pre(\sigma)\wedge pre(\sigma')\rightarrow\bot$
for each $\sigma\neq\sigma'\in\Gamma$. Together with exhaustivity,
this condition ensures that the product of $\Sigma{\scriptstyle \Gamma}$
and any $Ms\in X$ is a (single-)pointed Kripke model, i.e., that
the actual state after the updates is well-defined and unique.

Let $\Sigma{\scriptstyle \Gamma}$ be exhaustive and deterministic
over $X$ and let $Ms\in X$. Then the \textbf{product update} of
$Ms$ with $\Sigma{\scriptstyle \Gamma}$, denoted $Ms\otimes\Sigma{\scriptstyle \Gamma}$,
is the pointed Kripke model $(\left\llbracket M\Sigma\right\rrbracket ,R',\llbracket\cdot\rrbracket',s')$
with
\begin{eqnarray*}
\left\llbracket M\Sigma\right\rrbracket  & = & \left\{ (s,\sigma)\in\left\llbracket M\right\rrbracket \times\left\llbracket \Sigma\right\rrbracket :(M,s)\vDash pre(\sigma)\right\} \\
R' & = & \left\{ ((s,\sigma),(t,\tau)):(s,t)\in R_{i}\mbox{ and }(\sigma,\tau)\in\mathsf{R}_{i}\right\} ,\text{ for all }i\in\mathcal{I}\\
\left\llbracket p\right\rrbracket ' & = & \left\{ (s,\sigma)\!:\!s\in\left\llbracket p\right\rrbracket \!,post(\sigma)\nvDash\neg p\right\} \cup\left\{ (s,\sigma)\!:\!post(\sigma)\vDash p\right\} ,\text{ for all }p\in\Phi\\
s' & = & (s,\sigma):\sigma\in\Gamma\mbox{ and }Ms\vDash pre(\sigma)
\end{eqnarray*}

\noindent Call $\Sigma{\scriptstyle \Gamma}$ \textbf{closing over
$X$} if for all $x\in X,$ $x\otimes\Sigma{\scriptstyle \Gamma}\in X$.
With exhaustivity and deterministicality, this ensures that $\Sigma{\scriptstyle \Gamma}$
and $\otimes$ induce well-defined total map on $X$.

\subsection{Clean Maps on Modal Spaces}

Action models applied using product update yield natural maps on
modal spaces $\ms$. The class of maps of interest in the present
is thus the following:
\begin{defn}
Let $\ms$ be a modal space. A map $\boldsymbol{f}:\ms\rightarrow\ms$
is called \textbf{clean} if there exists a precondition finite, multi-pointed
action model $\Sigma{\scriptstyle \Gamma}$ closing, deterministic
and exhaustive over $X$ such that $\boldsymbol{f}(\boldsymbol{x})=\boldsymbol{y}$
iff $x\otimes\Sigma{\scriptstyle \Gamma}\in\boldsymbol{y}$ for all
$\boldsymbol{x}\in\ms$.
\end{defn}

\begin{rem}
Replacing $\ms$ with $\boldsymbol{X}_{D}$ for arbitrary descriptor
$D\subseteq\boldsymbol{\mathcal{L}}_{\Lambda}$ in the definition
of clean maps will not in general result in objects well-defined.
E.g.: Let $p$ and $q$ be atoms of $\mathcal{L}$ and let $D=\{\boldsymbol{p},\boldsymbol{\neg p}\}$.
Let $\Sigma{\scriptstyle \Gamma}$ have $\llbracket\Sigma\rrbracket=\Gamma=\{\sigma,\tau\}$
with $pre(\sigma)=q,pre(\tau)=\neg q$ and $post(\sigma)=\top$, $post(\tau)=p$.
Then for $x\models p\wedge q$ and $y\models p\wedge\neg q$, $y\in\boldsymbol{x}\in\boldsymbol{X}_{D}$,
but $y\otimes\Sigma{\scriptstyle \Gamma}\notin\boldsymbol{x\otimes\Sigma{\scriptstyle \Gamma}}$.
For $\Lambda$-representative descriptors, clean maps are, of course
well-defined 
\end{rem}

\medskip{}

Below, we show that clean maps are continuous with respect to the
Stone(-like) topology on $\ms$. For that proposition, we observe
that. By proposition ... and Lemma ...
\begin{rem}
By Proposition \ref{prop:same-topology} and Lemma \ref{lem:Repr.Descriptor}the
following analysis equally applies to the Stone(-like) topology on
$X_{D}$ for any $\Lambda$-represenative descriptor $D$.
\end{rem}

\begin{prop}
Any clean map $\boldsymbol{f}$ on the modal space $\boldsymbol{X}_{\boldsymbol{\mathcal{L}}_{\Lambda}}$
is total and well-defined.
\end{prop}

\begin{proof}
\noindent \noindent Clean maps are total on by the assumptions of the underlying
action model being closing and exhaustive. They are well-defined as
$\boldsymbol{f}(\boldsymbol{x})$ is independent of the choice of
representative for $\boldsymbol{x}$: If $x'\in\boldsymbol{x}$, then
$x'\otimes\Sigma{\scriptstyle \Gamma}$ and $x\otimes\Sigma{\scriptstyle \Gamma}$
are modally equivalent and hence define the same point in $\boldsymbol{X}_{\boldsymbol{\mathcal{L}}_{\Lambda}}$.
The latter follows as multi-pointed action models applied using product
update preserve bisimulation \cite{BaltagMoss2004}, which implies
modal equivalence.\medskip{}
\end{proof}
In general, the same clean map may be induced by several different
action models. In showing clean maps continuous, we will make use
of the following:
\begin{lem}
\label{lem:disjoint model representer} Let $\boldsymbol{f}:\ms\rightarrow\ms$
be a clean map based on $\Sigma{\scriptstyle \Gamma}$. Then there
exists an $\Sigma'{\scriptstyle \Gamma'}$ also inducing $\boldsymbol{f}$
such that for all $\sigma,\sigma'\in\left\llbracket \Sigma'\right\rrbracket $,
either $\models pre(\sigma)\wedge pre(\sigma')\rightarrow\bot$ or
$\models pre(\sigma)\leftrightarrow pre(\sigma')$.
\end{lem}

\begin{proof}
 Assume we are given any precondition finite, multi-pointed action
model $\Sigma{\scriptstyle \Gamma}$ deterministic over $X$ generating
$\boldsymbol{f}$. We construct an equivalent action model, $\Sigma'{\scriptstyle \Gamma'}$,
with the desired property.

For the preconditions, note that for every finite set of formulas
$S=\{\varphi_{1}\ldots\varphi_{n}\}$ there is some set formulas $\{\psi_{1},\ldots,\psi_{m}\}$
where all$\psi_{i},$ and $\psi_{j}$ are either logically equivalent
or mutually inconsisent such that each $\varphi\in S$ there is some
$J(\varphi)\subseteq\{1,\ldots,m\}$ such that $\vDash\bigvee_{k\in J(\varphi)}\psi_{k}\leftrightarrow\varphi$.
One suitable candidate for such a set is $\{\bigwedge_{k\leq n}\chi_{k}\colon\chi_{k}\in\{\varphi_{k},\neg\varphi_{k}\}\}$:
The disjunction of all conjunctions with $\chi_{k}=\varphi_{k}$ is
equivalent with $\varphi_{k}$. 

By assumption, $S=\{pre(\sigma)\colon\sigma\in\left\llbracket \Sigma\right\rrbracket \}$
is finite. Let $\{\psi_{1}\ldots\psi_{m}\}$ and $J(\varphi)$ be
as above. Construct $\Sigma'{\scriptstyle \Gamma'}$ as follows: For
every $\sigma\in\left\llbracket \Sigma\right\rrbracket $ and every
$\psi\in J(pre(\sigma))$, the set $\left\llbracket \Sigma'\right\rrbracket $
contains a state $e^{\sigma,\psi}$ with $pre(e^{\{\sigma,\psi\}})=\psi$
and $post(e^{\{\sigma,\psi\}})=post(\sigma)$. Let $R'$ be given
by $(e^{\sigma,\psi},e^{\sigma',\psi'})\in R'$ iff $(\sigma,\sigma')\in R$.
Finally, let $\Gamma'=\{e^{\{\sigma,\psi\}}\colon\sigma\in\Gamma\}$. 

The resulting multi-pointed action model $\Sigma'{\scriptstyle \Gamma'}$
is again precondition finite and deterministic over $X$ while having
either preconditions satisfying for all $\sigma,\sigma'\in\left\llbracket \Sigma'\right\rrbracket $,
either $\models pre(\sigma)\wedge pre(\sigma')\rightarrow\bot$ or
$\models pre(\sigma)\leftrightarrow pre(\sigma')$. Moreover, for
any $x\in X$, the models $x\otimes\Sigma{\scriptstyle \Gamma}$ and
$x\otimes\Sigma'{\scriptstyle \Gamma'}$\textbackslash{}in X $f(x)$
and $f'(x)$ are bisimilar witnessed by the relation connecting $(s,\sigma)\in\left\llbracket f(x)\right\rrbracket $
and $(s',e^{\sigma',\psi})\in\left\llbracket f'(x)\right\rrbracket $
iff $s=s'$ and $\sigma=\sigma'$. Hence, the maps $\boldsymbol{f},\ \boldsymbol{f'}:\ms\rightarrow\ms$
defined by $\boldsymbol{x}\rightarrow\boldsymbol{x\oplus\Sigma{\scriptstyle \Gamma}}$
and $\boldsymbol{x}\rightarrow\boldsymbol{x\oplus\Sigma'{\scriptstyle \Gamma'}}$
are the same.
\end{proof}

\subsection{Continuity of Clean Maps}

We show that the metrics introduced are reasonable with respect to
the analysis of dynamics modeled using clean maps by showing that
such a \textbf{continuous} in the induced topology:
\begin{prop}
\label{continuous} Any clean map $\boldsymbol{f}:\ms\rightarrow\ms$
is uniformly continuous in the metric space $(\ms,d_{w})$, for any
$d_{w}\in\mathcal{D}_{(X,D)}$ for $D$ $\Lambda$-representative.
\end{prop}

In the proof, we make use of the following lemma:
\begin{lem}
\label{lem:aux} Let $(\ms,d_{w})$ be a metric space, $d_{w}\in\mathcal{D}_{(X,\boldsymbol{\mathcal{L}}_{\Lambda})}$
for $D$ $\Lambda$-representative. Then 
\begin{enumerate}
\item \label{aux1}For every $\epsilon>0$, there are formulas $\chi_{1},\ldots,\chi_{l}\in\mathcal{L}$
such that every $x\in X$ satisfies some $\chi_{i}$, and whenever
$y\models\chi_{i}$ and $z\models\chi_{i}$ for some $i\leq l$, then
$d_{w}(\boldsymbol{y},\boldsymbol{z})<\epsilon$.
\item \label{aux2} For every $\varphi\in\mathcal{L}$, there is a $\delta$
such that for all $x\in X$, if $x\models\varphi$ and $d_{w}(\boldsymbol{x},\boldsymbol{y})<\delta$,
then\textup{ $y\models\varphi$}.
\end{enumerate}
\end{lem}

\begin{proof}
[Proof of Lemma \ref{lem:aux}]For 1., note that there is some $n>0$ for which $\sum_{k=n}^{\infty}w(\boldsymbol{\varphi}_{k})<\epsilon$.
For $j\in\{1,...,n-1\}$ pick some $\varphi_{j}\in\mathbf{\boldsymbol{\varphi}}_{j}$.
Let $J_{1},...,J_{2^{n-1}}$ be an enumeration of the subsets of $\{1,...,n-1\}$,
and let the formula $\chi_{i}$ be $\bigwedge_{j\in J_{i}}\varphi_{j}\wedge\bigwedge_{j\not\in J_{i}}\neg\varphi_{j}$
for each $i\in\{1,...,2^{n-1}\}$. Then each $x\in X$ must satisfy
$\chi_{i}$ for some $i$. Moreover, whenever $y\models\chi_{i}$
and $z\models\chi_{i}$, $d_{w}(\boldsymbol{y},\boldsymbol{z})=\sum_{k=1}^{\infty}w(\boldsymbol{\varphi}_{k})d_{k}(\boldsymbol{y},\boldsymbol{z})=\sum_{k=n}^{\infty}w(\boldsymbol{\varphi}_{k})d_{k}(\boldsymbol{y},\boldsymbol{z})<\epsilon$.
For 2., let $\varphi\in\mathcal{L}$ be given. Since D is representative,
there are $\{\boldsymbol{\psi}_{i}\}_{i\in I}\subseteq D$ such that
for all sets $S=\{\psi_{i}\}_{i\in J}\cup\{\neg\psi_{i}\}_{i\in I\setminus J}$
with $J\subseteq I$ either $\varphi$ or $\neg\varphi$ is $\Lambda$-entailed
by $S$. Then $\delta:=min_{i\in I}w(\boldsymbol{\psi_{i}})$ yields
the desired. 
\end{proof}
\begin{proof}
[Proof of Proposition \ref{continuous}]We show that $\boldsymbol{f}$ is uniformly continuous, using the
$\varepsilon$-$\delta$ formulation of continuity. 

Assume that $\epsilon>0$ is given. We have to find some $\delta>0$
such that for all $\boldsymbol{x},\boldsymbol{y}\in\ms$ $d_{w}(\boldsymbol{x},\boldsymbol{y})<\delta$
implies $d_{w}(\boldsymbol{f(x)},\boldsymbol{f(y)})<\epsilon$. By
Lemma \ref{lem:aux}.\ref{aux1}, there exist $\chi_{1},\ldots,\chi_{l}$
such that $f(x)\models\chi_{i}$ and $f(y)\models\chi_{i}$ implies
$d_{w}(\boldsymbol{f(x)},\boldsymbol{f(y)})<\epsilon$ and for every
$\boldsymbol{x}\in\ms$ there is some $i\leq l$ with $\boldsymbol{f(x)}\models\chi_{i}$.
We use $\chi_{1},\ldots,\chi_{l}$ to find a suitable $\delta$:\medskip{}

\noindent \textbf{Claim:} There is a function $\boldsymbol{\delta}:\mathcal{L}\rightarrow(0,\infty)$
such that for any $\varphi\in\mathcal{L}$, if $f(x)\models\varphi$
and $d_{w}(\boldsymbol{x},\boldsymbol{y})<\boldsymbol{\delta}(\varphi)$,
then $f(y)\models\varphi$.\medskip{}

\noindent Clearly, setting $\delta=\min\{\boldsymbol{\delta}(\chi_{i})\colon i\leq l\}$
yields a $\delta$ with the desired property. Hence the proof is completed
by a proof of the claim. The claim is shown by induction over the
complexity of $\varphi$. To be explicit, the function $\boldsymbol{\delta}:\mathcal{L}\rightarrow(0,\infty)$
will depend on the clean map $\boldsymbol{f}$ and the action model
$\Sigma{\scriptstyle \Gamma}$ it is based on. More precisely, $\boldsymbol{\delta}$
depends on the set $\{pre(\sigma)\colon\sigma\in\left\llbracket \Sigma\right\rrbracket \}$.
The below construction of $\boldsymbol{\delta}$ is a simultaneous
induction over all action models with the set of preconditions $\{pre(\sigma)\colon\sigma\in\left\llbracket \Sigma\right\rrbracket \}$.
By Lemma \ref{lem:disjoint model representer}, we can assume that
for all $\varphi\neq\psi\in\{pre(\sigma)\colon\sigma\in\left\llbracket \Sigma\right\rrbracket \}$,
it holds that $\vDash pre(\sigma)\wedge pre(\sigma')\rightarrow\bot$.
Wlog, assume all negations in $\varphi$ immediately precede atoms.\medskip{}

If $\varphi$ is an atom or negated atom: By Lemma \ref{lem:aux}.\ref{aux2},
there exists for any $\sigma\in\left\llbracket \Sigma\right\rrbracket $
some $\delta_{\sigma}$ such that whenever $x\models pre(\sigma)$
and $d_{w}(\boldsymbol{x},\boldsymbol{y})<\delta$$_{\sigma}$ we
also have that $y\models pre(\sigma)$. Likewise, there is some $\delta_{0}$
such that whenever $x\vDash\varphi$ and $d_{w}(\boldsymbol{x},\boldsymbol{y})<\delta_{0}$
we also have that $y\models\varphi$. By assumption, the set $\{pre(\sigma)\colon\sigma\in\left\llbracket \Sigma\right\rrbracket \}$
is finite. Let $S=\{\delta_{0}\}\cup\{\delta_{\sigma}\colon\sigma\in\left\llbracket \Sigma\right\rrbracket \}$.
We can thus set $\boldsymbol{\delta}(\varphi)=\min(S)$. To see that
this $\boldsymbol{\delta}$ is as desired, assume $f(x)\models\varphi$.
With $x=Ms$, there is a unique $\sigma\in\Gamma$ in the deterministic,
multi-pointed action model $(\Sigma,\Gamma)$ such that $(s,\sigma)$
is the designated state of $f(x)$. In particular, we have that $x\models pre(\sigma)$.
By our choice of $\boldsymbol{\delta}(\varphi)$, we get that $d_{w}(\boldsymbol{x},\boldsymbol{y})<\boldsymbol{\delta}(\varphi)$
implies $y\models pre(\sigma)$. For $y=Nt$, we thus have that $(t,\sigma)$
is the designated state of $f(Nt)$. Moreover, we have $x\vDash\varphi\Leftrightarrow y\vDash\varphi$.
Together, these imply that $f(Nt)\models\varphi$.\smallskip{}

If $\varphi$ is $\varphi_{1}\wedge\varphi_{2}$, set $\boldsymbol{\delta}(\varphi)=\min(\boldsymbol{\delta}(\varphi_{1}),\boldsymbol{\delta}(\varphi_{2})\,)$
To show that this is as desired, assume $f(x)\models\varphi_{1}\wedge\varphi_{2}$.
We thus have $f(x)\models\varphi_{1}$ and $f(x)\models\varphi_{2}$.
By induction, this implies that whenever $d_{w}(\boldsymbol{x},\boldsymbol{y})<\boldsymbol{\delta}(\varphi)$,
we have $f(y)\models\varphi_{1}$ and $f(y)\models\varphi_{2}$ and
hence $f(y)\models\varphi_{1}\wedge\varphi_{2}$.\smallskip{}

If $\varphi$ is $\varphi_{1}\vee\varphi_{2}$, set $\delta(\varphi)=\min(\boldsymbol{\delta}(\varphi_{1}),\boldsymbol{\delta}(\varphi_{2})\,)$
To show that this is as desired, assume $f(x)\models\varphi_{1}\vee\varphi_{2}$.
We thus have $f(x)\models\varphi_{1}$ or $f(x)\models\varphi_{2}$.
By induction, this implies that whenever $d_{w}(\boldsymbol{x},\boldsymbol{y})<\boldsymbol{\delta}(\varphi)$
we have $f(y)\models\varphi_{1}$ or $f(y)\models\varphi_{2}$ and
hence $f(y)\models\varphi_{1}\vee\varphi_{2}$.\smallskip{}

If $\varphi$ is $\Diamond\varphi_{1}$: By Lemma \ref{lem:aux}.\ref{aux1},
there are $\chi_{1},\ldots,\chi_{l}$ such that every $\boldsymbol{x}\in\ms$
satisfies some $\chi_{i}$ and whenever $z\models\chi_{i}$ and $z'\models\chi_{i}$
for some $i\leq l$ we have $d_{w}(z,z')<\boldsymbol{\delta}(\varphi_{1})$.

Now, let $F=\{\Diamond(pre(\sigma)\wedge\chi_{i})\colon\sigma\in\left\llbracket \Sigma\right\rrbracket ,i\leq l\}\cup\{pre(\sigma)\colon\sigma\in\left\llbracket \Sigma\right\rrbracket \}$.
By assumption, $F$ is finite. By Lemma \ref{lem:aux}.\ref{aux2},
for each $\psi\in F$ there is some $\delta_{\psi}$ such that $x\models\psi$
and $d_{w}(\boldsymbol{x},\boldsymbol{y})<\delta_{\psi}$ implies
$y\models\psi$. Set $\boldsymbol{\delta}(\varphi)=\min\{\delta_{\psi}\colon\psi\in F\}.$

To show that this is as desired, assume $f(x)\models\Diamond\varphi_{1}$
and let $y$ be such that $d_{w}(\boldsymbol{x},\boldsymbol{y})<\boldsymbol{\delta}(\varphi)$.
We have to show that $f(y)\models\Diamond\varphi_{1}.$ Let $x=Ms$
and let the designated state of $f(x)$ be $(s,\sigma)$. Since $f(x)\models\Diamond\varphi_{1}$,
there is some $(s',\sigma')$ in $\left\llbracket f(x)\right\rrbracket $
with $(s,\sigma)R(s',\sigma')$. In particular $x\models\Diamond(pre(\sigma')\wedge\chi_{i})$
for some $\sigma'\in\left\llbracket \Sigma\right\rrbracket $ and
$i\leq l$. Thus also $y\models\Diamond(pre(\sigma')\wedge\chi_{i})$.
Hence, with $y=Nt$, there is some $t'\text{\ensuremath{\in}}\left\llbracket y\right\rrbracket $
accessible from $y$'s designated state $t$ that satisfies $pre(\sigma')\wedge\chi_{i}.$
By determinacy and the fact that $\vDash pre(\sigma)\wedge pre(\sigma')\rightarrow\bot$
whenever $\varphi\neq\psi\in\{pre(\sigma)\colon\sigma\in\left\llbracket \Sigma\right\rrbracket \}$,
there is a unique$\tilde{\sigma}\in\Gamma$ with $pre(\tilde{\sigma})=pre({\sigma'})$.
Let $\Gamma'=\Gamma-\{\tilde{\sigma}\}\cup\{\sigma\}$ and let $f'$
be the model transformer induced by $\Sigma{\scriptstyle \Gamma'}$.
As $f'$ has the same set $\{pre(\sigma)\colon\sigma\in\left\llbracket \Sigma\right\rrbracket \}$
as $f$, our induction hypothesis applies to $f'$ . Consider the
models $Ms'$ and $Nt'$. We have that $Ms'\models\chi_{i}$ and $Nt'\vDash\chi_{i}$
jointly imply $d_{w}(Ms',Nt')<\boldsymbol{\delta}(\varphi_{1})$ which,
in turn, implies that $f'(Ms')\models\varphi_{1}$ iff $f'(Nt')\models\varphi_{1}$.
In particular, we obtain that $\left\llbracket f(y)\right\rrbracket ,(t',\sigma')\models\varphi_{1}$.
Since $(t,\sigma)R(t',\sigma')$ this implies that $f(y)\models\Diamond\varphi_{1}.$\smallskip{}

If $\varphi$ is $\Box\varphi_{1}$: The construction is similar to
the previous case. We only give the relevant differences. Again, there
are some$\chi_{1},\ldots,\chi_{l}$ such that every $\boldsymbol{x}\in\ms$
satisfies some $\chi_{i}$ and whenever $z\vDash\chi_{i}$ and $z'\vDash\chi_{i}$
for some $i\leq l$ we have $d_{w}(\boldsymbol{z},\boldsymbol{z}')<\boldsymbol{\delta}(\varphi_{1})$.

Now, let $R=\{pre(\sigma)\wedge\chi_{i}\colon\sigma\in\left\llbracket \Sigma\right\rrbracket ,i\leq l\}$
and let $F=\{\Box(\bigvee_{k\in J}k)\colon J\subseteq R\}\cup\{pre(\sigma)\colon\sigma\in\left\llbracket \Sigma\right\rrbracket \}$.
Again, $F$ is finite and for each $\psi\in F$ there is some $\delta_{\psi}$
such that $x\models\psi$ and $d_{w}(\boldsymbol{x},\boldsymbol{y})<\delta_{\psi}$
implies $y\models\psi$. Set $\delta(\varphi)=\min\{\delta_{\psi}\colon\psi\in F\}.$

To show that this is as desired, assume $f(x)\models\Box\varphi_{1}$
and let $y$ be such that $d_{w}(\boldsymbol{x},\boldsymbol{y})<\boldsymbol{\delta}(\varphi)$.
We have to show that $f(y)\models\Box\varphi_{1}.$ Let $y=Nt$ ,
let $(t,\sigma)$ be the designated state of $f(y)$ and assume there
is some $(t',\sigma')$ in $\left\llbracket f(y)\right\rrbracket $
with $(t,\sigma)R(t',\sigma')$. We have to show that $\varphi_{1}$
holds at $(t',\sigma')$. To this end, note that by construction,
$t'$ satisfies $pre(\sigma')\wedge\chi_{i}$, for some $i\leq l$.
By the choice of $\boldsymbol{\delta}(\varphi)$, there is some $s'\in\left\llbracket x\right\rrbracket $
with $sRs'$ (for $x=Ms$) that also satisfies $pre(\sigma')\wedge\chi_{i}$.
Hence $(s',\sigma')$ is in $\left\llbracket f(x)\right\rrbracket $
and $(s,\sigma)R(s',\sigma')$. By assumption we have $(s',\sigma')\models\varphi_{1}$
and by an argument similar to the last case we get $(t',\sigma')\models\varphi_{1}$.
Hence $f(y)\models\square\varphi_{1}$.
\end{proof}
\begin{cor}
Any clean map $\boldsymbol{f}:\ms\rightarrow\ms$ is continuous with
respect to the Stone(-like) topology \textbf{\emph{$\mathcal{T}_{\boldsymbol{\mathcal{L}}_{\Lambda}}$}}.
\end{cor}

\subsubsection*{{\small{}Acknowledgments.}}

{\small{}}%
{\small \par}

{\small{}}%

{\small{}The contribution of R.K. Rendsvig was funded by the Swedish
Research Council through the framework project `Knowledge in a Digital
World' (Erik J. Olsson, PI) and The Center for Information and Bubble
Studies, sponsored by The Carlsberg Foundation. We thank Kristian
Knudsen Olesen for his thorough reading and invaluable comments, Alexandru
Baltag, Johan van Benthem, Nick Bezhanishvili, Paolo Galeazzi, Hannes
Leitgeb, Olivier Roy and the participants of LogiCIC 2015 and 2016
(Amsterdam), CADILLAC 2016 (Copenhagen), The von Wright Symposium
(2016, Helsinki), Higher Seminar in Theoretical Philosophy (2016 and
2017, Lund), Tsinghua-Bayreuth Logic Workshop 2016 (Beijing), and
a session of the MCMP Logic Seminar 2017 (Munich) for valuable comments
and discussion.}

{\small \par}

\DeclareRobustCommand{\VAN}[2]{#2}%

\bibliographystyle{abbrv}

\end{document}